\documentclass[a4paper]{article}
\usepackage{amsmath,amssymb,amscd,amsthm}
\usepackage{epic,eepic}
\usepackage[dvips]{graphics}
\usepackage[usenames]{color}
\usepackage{comment}

\DeclareFontEncoding{OT2}{}{}
\DeclareFontSubstitution{OT2}{wncyr}{m}{n}
\DeclareSymbolFont{cyss}{OT2}{wncyss}{m}{n}
\DeclareSymbolFont{cyr}{OT2}{wncyr}{m}{n}
\DeclareMathSymbol{\sh}{\mathbin}{cyss}{`x}

\definecolor{gray}{gray}{0.5}

\newcommand{\End}{\operatorname{End}}
\newcommand{\ad}{\operatorname{ad}}
\newcommand{\PGL}{\operatorname{PGL}}
\newcommand{\Lie}{\operatorname{Lie}}
\newcommand{\Li}{\operatorname{Li}}

\newcommand{\reg}{\operatorname{reg}}

\newcommand{\id}{\operatorname{id}}
\newcommand{\sgn}{\operatorname{sgn}}
\newcommand{\Image}{\operatorname{Im}}

\newcommand{\C}{{\mathbf C}}
\newcommand{\R}{{\mathbf R}}
\newcommand{\Q}{{\mathbf Q}}
\newcommand{\Z}{{\mathbf Z}}

\newcommand{\F}{{\mathbf F}}
\newcommand{\bP}{{\mathbf P}}
\newcommand{\bunit}{{\mathbf I}}
\newcommand{\bnull}{{\mathbf 1}}

\newcommand{\fX}{{\mathfrak X}}

\newcommand{\fM}{{\mathfrak M}}
\newcommand{\fA}{{\mathfrak A}}

\newcommand{\cB}{{\mathcal B}}

\newcommand{\cL}{{\mathcal L}}
\newcommand{\cM}{{\mathcal M}}
\newcommand{\cU}{{\mathcal U}}
\newcommand{\tcU}{{\widetilde{\mathcal U}}}
\newcommand{\cW}{{\mathcal W}}
\newcommand{\hcL}{{\widehat{\mathcal L}}}
\newcommand{\tcL}{{\widetilde{\mathcal L}}}

\newcommand{\ds}{\displaystyle}

\theoremstyle{definition}
\newtheorem{thm}{Theorem}

\newtheorem{prop}[thm]{Proposition}
\newtheorem{lem}[thm]{Lemma}

\def\theenumi{\arabic{enumi}}
\def\labelenumi{(\theenumi)}

\title{\huge Connection Problem of Knizhnik-Zamolodchikov Equation on Moduli Space $\cM_{0,5}$\\[0.5cm]}
\author{OI, Shu\thanks{Department of Mathematics, School of Fundamental Sciences and 
Engineering, Faculty of Science and Engineering, Waseda University. \endgraf \hspace{1em}3-4-1, Okubo, Shinjuku-ku, 
Tokyo 169-8555, Japan. \endgraf \hspace{1em}e-mail: {\tt shu\_oi@toki.waseda.jp}}
\and UENO, Kimio\thanks{Department of Mathematics, School of Fundamental Sciences 
and Engineering, Faculty of Science and Engineering, Waseda University. \endgraf \hspace{1em}3-4-1, Okubo, Shinjuku-ku, 
Tokyo 169-8555, Japan. \endgraf \hspace{1em}e-mail: {\tt uenoki@waseda.jp}
}}
\date{}

\begin{document}


\maketitle

\begin{abstract}
\noindent 
In this article, we consider the connection problem of the KZ (Knizhnik-Zamolodchikov) equation on the 
moduli space $\cM_{0,5}$, and show that the connection matrices are expressed in terms of the Drinfel'd 
associator. As the compatibility condition on the connection problem, we have the pentagon relation 
for the Drinfeld associators. As an application of the connection problem, we derive the five term 
relation for dilogarithms. 
\end{abstract}

\section{Introduction}

In this article, we consider the connection problem of the KZ equation on the 
moduli space $\cM_{0,5}$, and show that the connection matrices are expressed in terms of the Drinfel'd 
associators. As the compatibility condition on the connection problem, we have the pentagon relation 
for the Drinfeld associators. This story was already proved by Drinfel'd in his famous paper \cite{Dr} 
(see also Wojtkowiak's paper \cite{W}), but our proof is completely different from theirs. Our proof is 
based upon the decomposition theorem for fundamental solutions of the KZ equation. 
As an application of the connection problem, we derive the five term 
relation for dilogarithms. 

This paper is organized as follows: In section 2, we introduce the KZ equation on the moduli space 
$\cM_{0,n}$. In the cubic coordinate system of the space, it is represented as the KZ equation of 
$n-3$ variables. In section 3, we particularly consider the KZ equation of one variable which is denoted
by $z$, and show that the Drinfel'd associator $\varPhi_{\rm KZ}$ gives the connection matrix between $\cL(z)$, 
which is the fundamental solution of normalized at $z=0$ of the equation, and $\cL^{(1)}(z)$, 
which is the fundamental solution normalized at $z=1$ of the equation; 
\begin{align}\label{eq:connection}
          \cL(z)=\cL^{(1)}(z)\varPhi_{\rm KZ}.
\end{align}
\noindent 
In section 4, the KZ equation of two variables, which are denoted by $(z_1,z_2)$, will be considered. 
The fundamental solution $\cL(z_1,z_2)$ normalized at $(z_1,z_2)=(0,0)$ of the equation is shown to decompose 
as follows: 
\begin{align}\label{decomposition}
     \cL(z_1,z_2) &= \cL_{2 \otimes 1}^{(2)}(z_1,z_2)\cL_{2 \otimes 1}^{(1)}(z_1) \\
                  &= \cL_{1 \otimes 2}^{(1)}(z_1,z_2)\cL_{1 \otimes 2}^{(2)}(z_2) \nonumber 
\end{align}
where $\cL_{i \otimes j}^{(j)}(z_j)$ is the fundamental solution normalized at $z_j=0$ of the KZ  equation of the one variable $z_j$,
and $\cL_{i \otimes j}^{(j)}(z_1,z_2)$ is the fundamental solution normalized at $z_i=0$
of the Schlesinger type equation of the variable $z_i$. 
We will refer to \eqref{decomposition} as ``the decomposition theorem''. In section 5, we will consider the connection 
problem of solutions of the KZ equation on $\cM_{0,5}$. We fix a pentagon in $\overline{\cM}_{0,5}(\R)$, 
which denotes the real points of a certain compactification $\overline{\cM}_{0,5}$ of $\cM_{0,5}$, 
as below: 

\hspace{20mm}
\begin{picture}(0,4.5)(0,-0.5)
\footnotesize
\put(0,0){\scalebox{0.5}{\includegraphics{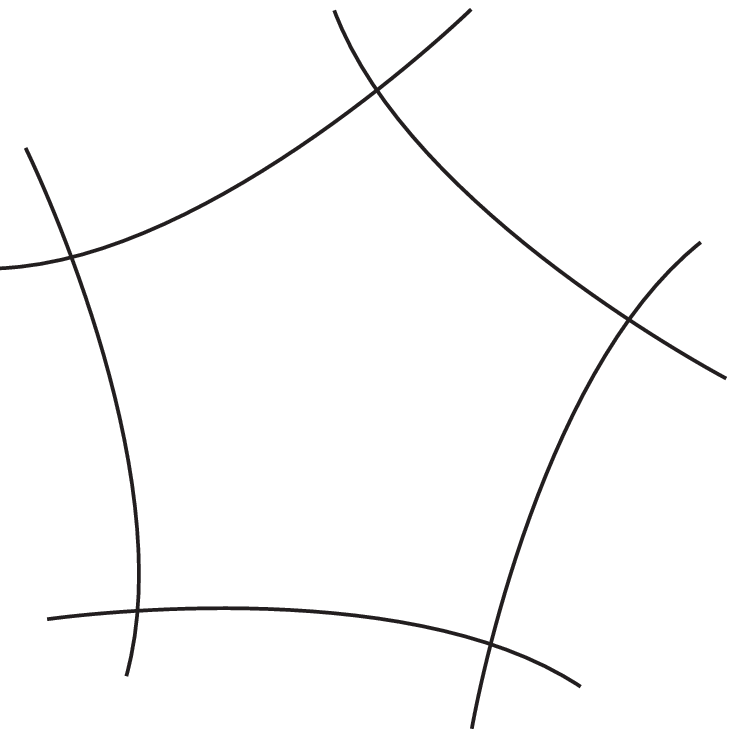}}}
\put(1.2,0.3){$D_{23}$}
\put(2.7,1.0){$D_{14}$}
\put(2.4,2.7){$D_{35}$}
\put(0.7,3.0){$D_{12}$}
\put(-0.1,1.4){$D_{45}$}
\put(0.7,0.7){$P_0$}
\put(2.1,0.6){$P_1$}
\put(2.6,2.0){$P_2$}
\put(1.6,2.8){$P_3$}
\put(0.5,2.1){$P_4$}
\end{picture}

\noindent
To each vertex $P_{\alpha}$, we associate the cubic coordinate $(z_1^{(\alpha)},z_2^{(\alpha)})$, and 
the KZ equation of the two variable $(z_1^{(\alpha)},z_2^{(\alpha)})$. Let $\cL^{(\alpha)}$ be 
the fundamental solution normalized at the vertex $P_{\alpha}$ of the equation. Then one can show that 
\begin{align}\label{eq:connection_pentagon}
    \cL^{(\alpha)}=\cL^{(\alpha+1)}\varPhi^{(\alpha)}_{\rm KZ} \quad (\alpha=0,1,2,3,4 \pmod{5})
\end{align}
where $\varPhi^{(\alpha)}_{\rm KZ}$ is the Drinfel'd associator attached to the vertex $P_{\alpha}$. 
The relation \eqref{eq:connection_pentagon} is proved by applying the decomposition theorem 
\eqref{decomposition}. 
As the compatibility condition for $\eqref{eq:connection_pentagon}$, we obtain the pentagon 
relation of the Drinfel'd associators
\begin{align}\label{eq:pentagon_relation}
\varPhi^{(4)}_{\rm KZ}\;\varPhi^{(3)}_{\rm KZ}\;
\varPhi^{(2)}_{\rm KZ}\;\varPhi^{(1)}_{\rm KZ}\;\varPhi^{(0)}_{\rm KZ}=\bunit
\end{align}

In section 6, as an application of the connection problem, we derive the five term relation for dilogarithms 
\begin{align}\label{eq:5_term_relation}
\Li_2(z_1z_2)= \Li_2 & \left(\frac{-z_1(1-z_2)}{1-z_1}\right)  + 
\Li_2\left(\frac{-z_2(1-z_1)}{1-z_2}\right)  \nonumber \\
& + \Li_2(z_1)+\Li_2(z_2)+\frac{1}{2}\log^2\left(\frac{1-z_1}{1-z_2}\right). 
\end{align}

In \cite{OU}, we show that the decomposition theorem $\eqref{decomposition}$ is equivalent to 
the generalized harmonic product relations of hyperlogarithms of the type $\cM_{0,5}$, which 
contain the harmonic product of multiple polylogarithms.

\paragraph{Acknowledgment}
\noindent 
The first author is partially supported by Waseda University Grant for Special
Research Projects No. 2010B-200, 2011B-095.
The second author is partially supported by JSPS Grant-in-Aid No. 19540056, 22540035.

\section{The KZ equation on the moduli space $\cM_{0,n}$}
\subsection{Definition of the moduli space and the KZ equation}

The moduli space $\cM_{0,n}$ is by definition  \cite{Y}
\begin{align}
    \cM_{0,n}=\PGL(2,\C) \; \Big\backslash \; \F_n(\bP^1).
\end{align}
Here $\F_n(\bP^1)$ is the configuration space of $n$ points on $\bP^1=\bP^1(\C)$, 
\begin{align*}
    \F_n(\bP^1)&=\{(x_1,\ldots,x_n) \in (\bP^1)^n \;|\; x_i\neq x_j\;\; (i\neq j)\}\\
&=(\bP^1)^n \setminus \Delta
\end{align*}
where $\Delta$ is the thick diagonal of $(\bP^1)^n$, 
\begin{align*}
    \Delta&=\{(x_1,\ldots,x_n) \in (\bP^1)^n \; | \; \exists i\neq j \text{ s.t. } x_i=x_j\}.
\end{align*}
The KZ equation(KZE) on $\cM_{0,n}$ is a Pfaffian system defined by 
\begin{align}\tag{KZE}\label{KZE}
dG=\varOmega G,  \quad \varOmega = \sum_{1\le i<j\le n}\omega_{ij}\varOmega_{ij}, \quad \omega_{ij}=d\log(x_i-x_j).
\end{align}
Here $\{\varOmega_{ij}\}$ are generators of the Lie algebra $\fX$ of the fundamental group of 
$\pi_1(\cM_{0,n})$, 
\begin{equation*}
\fX=\Lie(\pi_1(\cM_{0,n})).
\end{equation*}
This Lie algebra is defined through the lower central series of $\pi_1(\cM_{0,n})$, 
\begin{align*}
     \pi_1(\cM_{0,n})=G_1\supset G_2 \supset G_3 \supset \cdots.
\end{align*}
We set
\begin{align*}
     \fX=\bigoplus_{j=1}^\infty \C \; G_j/G_{j+1}.
\end{align*}
We should note that the fundamental group $\pi_1(\F_n(\bP^1))$ of the configuration space $\F_n(\bP^1)$ 
is isomorphic to 
\begin{align*}
   \pi_1(\F_n(\bP^1)) \cong \pi_1(\cM_{0,n}) \times \Z \slash 2\Z
\end{align*}
so that $\Lie(\pi_1(\F_n(\bP^1)))$ is isomorphic to $\Lie(\pi_1(\cM_{0,n}))$. 

The generators $\{\varOmega_{ij}\}_{1\le i,j\le n}$ of $\fX$ satisfy the following relations \cite{I}:
\begin{align}
&\varOmega_{ij}=\varOmega_{ji}, \quad \varOmega_{ii}=0, \tag{IH1}\label{IH1}\\
&\sum_{j=1}^n \varOmega_{ij}=0 \quad (\forall i ), \tag{IH2}\label{IH2}\\
&[\varOmega_{ij},\varOmega_{kl}]=0 \quad (\{i,j\} \cap \{k,l\} = \emptyset). \tag{IH3}\label{IH3}
\end{align}
These relations are referred to as the Ihara relations. $\fX$ is a Lie algebra generated 
by $\{\varOmega_{ij}\}_{1\le i,j\le n}$ with the fundamental relations \eqref{IH1} $\sim$ \eqref{IH3}.
The relation \eqref{IH1} is rather convention of notations. From \eqref{IH2} and \eqref{IH3}, we have 
\begin{align}\label{IH4}
      [\varOmega_{ij}+\varOmega_{jk},\varOmega_{ik}]=0    \quad (\#\{i,j,k\}=3).
\end{align}

The one forms $\{\omega_{ij}\}_{1\le i<j\le n}$ satisfy the Arnold relations \cite{A}:
\begin{align}\tag{AR}\label{AR}
      \omega_{ij} \wedge \omega_{ik} + \omega_{ik} \wedge \omega_{jk} + 
      \omega_{jk} \wedge \omega_{ij} = 0
\end{align}
for $i<j<k$. It is known that $\eqref{AR}$ are unique  nontrivial relations of degree 2 
which hold among these one forms. 

From the Ihara relations \eqref{IH4} and the Arnold relations \eqref{AR}, 
we can deduce that 
\def\theenumi{\alph{enumi}}
\def\labelenumi{(\theenumi)}
\begin{enumerate}
\item the system \eqref{KZE} is regular as $x_i=\infty,\; |x_j|<\infty\;(j\ne i)$.
\item the system \eqref{KZE} is integrable, namely $\varOmega$ satisfy
\begin{align*}
     \varOmega \wedge \varOmega=0.
\end{align*}
\item the system \eqref{KZE} is $\PGL(2, \C)$ invariant.
\end{enumerate}
\def\theenumi{\arabic{enumi}}
\def\labelenumi{(\theenumi)}

\noindent 
Hence the system \eqref{KZE} can be regarded as an integrable system on $\cM_{0,n}$.

\subsection{Cubic coordinate system on $\cM_{0,n}$ and KZ equation of $n-3$ variables}

Let $\ds r(i,j;k,l)=\frac{(x_i-x_k)(x_j-x_l)}{(x_i-x_l)(x_j-x_k)}$ be a cross ratio. 
Put
\begin{align}
    y_i=r(i,n-1;n-2,n)=\frac{x_i-x_{n-2}}{x_i-x_n}\frac{x_{n-1}-x_{n}}{x_{n-1}-x_{n-2}} 
                                                \in \bP^1\setminus\{0,1,\infty\}
\end{align}
for $i=1,\ldots,n-3$. This defines a coordinate system which we call the simplicial coordinate system on 
$\cM_{0,n}$. Now we introduce $(z_1,\ldots,z_{n-3})$ by
\begin{align}
       &y_i=z_1\cdots z_i \quad (1 \leq i \leq n-3)\\
             \Longleftrightarrow & 
          \begin{cases}
          z_1=y_1=r(1,n-1;n-2,n),\\
          z_2= \dfrac{y_2}{y_1}=r(2,1;n-2,n),\\
             \vdots\\
          z_{n-3}=\dfrac{y_{n-3}}{y_{n-4}}=r(n-3,n-4;n-2,n).
          \end{cases}
\end{align}
This is also a coordinate system on $\cM_{0,n}$ which we call the cubic coordinate system. These two 
coordinate system were introduced by F. Brown in \cite{B}.

In the cubic coordinate system on $\cM_{0,n}$, \eqref{KZE} is represented as follows:
\begin{align}\label{eq:KZ_n-3}
    & dG= \varOmega G, \nonumber \\
    & \varOmega= \sum_{k=1}^{n-3}\xi_kX_k+\sum_{k=1}^{n-3}\xi_{kk}X_{kk} +\sum_{1\le i<j\le n-3}\xi_{ij}X_{ij},
\end{align}
where
\begin{align}
        \xi_k=\frac{dz_k}{z_k},\quad \xi_{kk}=\frac{dz_k}{1-z_k},
        \quad \xi_{ij}=\frac{d(z_i\cdots z_j)}{1-(z_i\cdots z_j)} \quad (i<j)
\end{align}
and
\begin{align*}
    & X_k=\sum_{k\le i<j\le n-2}\varOmega_{ij} \quad (1\le k \le n-3),\quad 
      X_{11}=-\varOmega_{1,n-1},\quad X_{kk}=-\varOmega_{k-1,k},\\
    & X_{1j}=-\varOmega_{j,n-1} \quad (2\le j \le n-3),\quad X_{ij} =-\varOmega_{i-1,j} \quad (2\le i<j\le n-3).
\end{align*}
We refer to this system as the KZ equation of $n-3$ variables.

\section{KZE of one variable}
\subsection{Definition of KZE of one variable}

By the cubic coordinate system $\ds z=r(1,3;2,4)=\frac{(x_1-x_2)}{(x_1-x_4)}\frac{(x_3-x_4)}{(x_3-x_2)}$, 
the moduli space $\cM_{0,4}$ is homeomorphic to $\bP^1 \setminus \{0,1,\infty\}$;
\begin{align*}
 \cM_{0,4}\overset{\sim}{\longrightarrow} \bP^1\setminus\{0,1,\infty\} \ : \ 
 [x_1,x_2,x_3,x_4] \mapsto  z
\end{align*}
where $[x_1,\dots,x_4]$ stands for the $\PGL(2, \C)$ orbit of $(x_1,\ldots,x_4)$.  
The KZ equation on $\cM_{0,4}$ is represented as
\begin{align}\tag{KZE1}\label{KZE1}
\frac{dG}{dz}=\left(\frac{X_0}{z}+\frac{X_1}{1-z}\right)G
\end{align}
where $X_0=\varOmega_{12},\; X_1=-\varOmega_{13}$. This is the KZ equation of one variable.

In this case, $\fX$ is a free Lie algebra generated by $X_0, X_1$, which we denote by 
\begin{align}
           \fX=\C\{X_0,X_1\},
\end{align}
and the universal enveloping algebra $\cU(\fX)$ is a non-commutative polynomial algebra of $X_0,X_1$; 
\begin{align}
         \cU=\cU(\fX)=\C\langle X_0,X_1\rangle. 
\end{align}

\subsection{Shuffle algebra of one forms}

Put
\begin{align}
         \xi_0=\frac{dz}{z},\quad \xi_1=\frac{dz}{1-z}.
\end{align}
They satisfy the Arnold relation $\xi_0 \wedge \xi_1=0$. We set
\begin{align}
       S=S(\xi_0,\xi_1)=\big(\C\langle\xi_0,\xi_1\rangle, \sh, \bnull\big)
\end{align}
where $\bnull$ stands for the empty word in S and $\sh$ denotes the shuffle product, 
which is recursively defined by
\begin{gather}
w \sh \bnull = \bnull \sh w = w,\\
(\xi_{i_1} w_1) \sh (\xi_{i_2} w_2)=\xi_{i_1} (w_1 \sh (\xi_{i_2} w_2)) + \xi_{i_2} ((\xi_{i_1} w_1) \sh w_2).
\end{gather}
$S$ is an associative commutative algebra with the unit elemnt $\bnull$. This is called a shuffle algebra
generated by $\xi_0, \xi_1$.

$\cU$ is a non-commutative and cocomutative Hopf algebra, and $S$ is a commutative Hopf algebra. 
As Hopf algebras, $\cU$ and $S$ are dual to each other.

To a word $w=\xi_{i_1}\cdots\xi_{i_r} \in S$, we can associate an iterated integral 
$\ds \int_a^b w$ as follows \cite{C1,Ha}:
\begin{align}
             \int_a^b \bnull=1,\qquad
             \int_a^b w=\int_a^b \xi_{i_1}(t)\int_a^t \xi_{i_2}\cdots\xi_{i_r}.
\end{align}
Then we have
\begin{align}
        \int_a^b w_1\int_a^b w_2=\int_a^b (w_1\sh w_2).
\end{align}
For $w=\xi_0^{k_1-1}\xi_1\cdots\xi_0^{k_r-1}\xi_1$, we can put an iterated integral $\ds \int_0^z w$, 
which defines the multiple polylogarithm of one variable (MPL):
\begin{align}
          \Li_{k_1,\ldots,k_r}(z):=\int_0^z\xi_0^{k_1-1}\xi_1\cdots\xi_0^{k_r-1}\xi_1.
\end{align}
This is a many-valued function on $\cM_{0,4}$ and has a Taylor expansion
\begin{align}
\Li_{k_1,\ldots,k_r}(z)=\sum_{n_1>n_2\cdots>n_r>0}\frac{z^{n_1}}{n_1^{k_1}\cdots n_r^{k_r}} 
                                            \quad (|z|<1).
\end{align}
In particular if $k_1\ge 2$, we can define an iterated integral 
$\ds \int_0^1\xi_0^{k_1-1}\xi_1\cdots\xi_0^{k_r-1}\xi_1$. This is the multiple zeta value (MZV):
\begin{align}
            \zeta(k_1,\ldots,k_r)=\int_0^1\xi_0^{k_1-1}\xi_1\cdots\xi_0^{k_r-1}\xi_1
                =\sum_{n_1>n_2\cdots>n_r>0}\frac{1}{n_1^{k_1}\cdots n_r^{k_r}}.
\end{align}

Let us introduce shuffle subalgebras such as 
\begin{align}
         S^0=\C\bnull\oplus S\xi_1, \quad 
         S^{10}=\C\bnull\oplus \xi_0S\xi_1.
\end{align}
$S^0$ is a subalgebra spanned by words ending with other than $\xi_0$, 
and $S^{10}$ is a subalgebra spanned by words beginning with other than $\xi_1$ 
and ending with other than $\xi_0$.

\begin{prop}[\cite{R}]\label{prop:reutanuer}
We have
\begin{align}\label{eq:reutanuer}
           S=S^0[\xi_0]=S^{10}[\xi_1,\xi_0].
\end{align}
\end{prop}

Let us define the regularization maps as follows: 
\begin{align}
      & \reg^0: S=S^0[\xi_0] \to S^0, \quad u=\sum w_j\sh \xi_0^j \mapsto w_0 \ (w_j \in S^0) \\
      & \reg^{10}: S=S^{10}[\xi_1,\xi_0] \to S^{10}, \quad 
         u=\sum \xi_1^i\sh w_{ij}\sh \xi_0^j \mapsto w_{00} \ (w_{ij} \in S^{10}).
\end{align}
From Proposition \ref{prop:reutanuer}, these maps turn out to be well-defined and homomorphisms.
Furthermore we can also define the MPL map and the MZV map as follows:
\begin{itemize}
\item MPL map $\Li(\bullet;z): S \to \C$: \ For $u=\sum w_j\sh \xi_0^j \quad (w_j \in S^0)$, we put 
\begin{align}
              \Li(u;z):=\sum\Li(w_j;z)\frac{\log^j z}{j!}.
\end{align}
Here for a word $w=\xi_0^{k_1-1}\xi_1\cdots\xi_0^{k_r-1}\xi_1 \in S^0$, $\Li(w;z)$ is 
the MPL $\Li_{k_1,\ldots,k_r}(z)$.
\item MZV map $\zeta: S^{10} \to \C$: \ For $w=\xi_0^{k_1-1}\xi_1\cdots\xi_0^{k_r-1}\xi_1 \in S^{10}$
(i.e. $k_1\ge 2$), we put
\begin{align}
      \zeta(w)=\zeta(k_1,\ldots,k_r).
\end{align}
\end{itemize}

\subsection{Fundamental solutions of KZE1}

\begin{prop}\label{prop:KZE1}
The equation \ref{KZE1} $\ds \frac{dG}{dz}=\Big(\frac{X_0}{z}+\frac{X_1}{1-z}\Big)G$ 
has a unique solution $\cL(z)$ satisfying the following asymptotic condition;
\begin{align*}
     & \cL(z)=\hcL(z)z^{X_0}, \\
     & \hcL(z) \ \text{is holomorphic at $z=0$ and $\hcL(0)=\bunit$}.
\end{align*}
\end{prop}
\begin{proof}
Let us express $\hcL(z)$ such as $\hcL(z)=\sum_{s=0}^{\infty} \hcL_s(z)$ where 
$\hcL_s(z) \in \cU_s$ is the homogeneous component of the degree $s$. Then it is easy to show that  
$\hcL_s(z)$ satisfies the following recursive relation;
\begin{equation}
        \frac{d\hcL_{s+1}(z)}{dz}=\frac{1}{z}[X_0,\hcL_s(z)]+\frac{1}{1-z}X_1\hcL_s(z).
\end{equation}
Hence we have $\hcL_0(z)=\bunit$ and for $s\ge 1$,
\begin{align}\label{eq:MPLs0}
     \hcL_s(z) &=\int_0^z \Big(\xi_0\otimes \ad(X_0)+\xi_1\otimes\mu(X_1)\Big)^s(\bnull\otimes\bunit)
                                            \nonumber \\
               &=\!\!\!\sum_{k_1+\cdots+k_r=s}\!\!\!\Li_{k_1,\ldots,k_r}(z)\;\ad(X_0)^{k_1-1}\mu(X_1)
                              \cdots\ad(X_0)^{k_r-1}\mu(X_1)(\bunit),
\end{align}
where $\ad(X_0)(A)=[X_0,A],\;  \mu(X_1)(A)=X_1A$ for $A \in \cU$.
Hence $\cL(z)$ exists and is unique.
\end{proof}

We call $\cL(z)$ a fundamental solution normalized at $z=0$ of \eqref{KZE1}.
We should observe that $\cL(z)$ and $\hcL(z)$ are grouplike elements of 
$\tcU=\C\langle\langle X_0,X_1\rangle\rangle$ (a non-commutative formal power series ring 
of $X_0,X_1$), and that $\hcL(z)$ is single valued and holomorphic in $D=\C \setminus \{z=x\,|\, 1 \le x\}$. 

Let us find out another representation of $\cL(z)$. We set
\begin{align*}
        G(z)=\sum_{w \in S}\Li(w;z)W
\end{align*}
where in the sum of the right hand side, $w$ ranges over the set of words in $S$, and 
for $w=\xi_{i_1}\cdots\xi_{i_r}$, we put $W=X_{i_1}\cdots X_{i_r}$. They are dual basis to each other. 

We can show the following lemma \cite{HPH}, \cite{Ok}.

\begin{lem} For any $u \in S$, one has 
\begin{align}
\begin{cases}\label{eq:diffeq_Limap1}
      \dfrac{d}{dz}\Li(\xi_0u;z)=\dfrac{\Li(u;z)}{z},\\
      {}                                            \\
      \dfrac{d}{dz}\Li(\xi_1u;z)=\dfrac{\Li(u;z)}{1-z}.
\end{cases}
\end{align}
Hence $G(z)$ is a solution of \eqref{KZE1}.
\end{lem}
\begin{proof}
Putting $u=v \sh \xi_0^r \ (v \in S^0, \ r \geq 0)$, we prove \eqref{eq:diffeq_Limap1} by 
induction on $r$. If $r=0$, then $u=v \in S^0$. This case is already proved because we have 
\begin{align}\label{eq:diifeq_mpl1}
    \frac{d}{dz} \Li_{k_1,k_2,\ldots,k_r}(z) =
       \begin{cases}
           \dfrac{1}{z}\Li_{k_1-1,k_2,\ldots,k_r}    &   (k_1 \ge 2),  \\
                                                                     \\
           \dfrac{1}{1-z} \Li_{k_2,\ldots,k_r}(z)  &  ( k_1=1).
       \end{cases}
\end{align}
Now assume that, for $r-1 \ (r \ge 1)$, \eqref{eq:diffeq_Limap1} holds. Since 
\begin{align*}
          \xi_iv \sh \xi_0^r = \xi_i(v \sh \xi_0^r) + \xi_0(\xi_iv \sh \xi_0^{r-1}),
\end{align*}
applying $\Li(\bullet ; z)$ to this relation, we have 
\begin{align*}
      \Li(\xi_i(v \sh \xi_0^r);z)=\Li(\xi_iv;z)\frac{\log^rz}{r!} - \Li(\xi_0(\xi_iv \sh \xi_0^{r-1});z).
\end{align*}
By the assumption of induction, the right hand side above reads 
\begin{align*}
        & \frac{d}{dz}\{\mathrm{RHS}\}= \frac{d}{dz}\Li(\xi_iv;z)\cdot\frac{\log^rz}{r!}+
                   \frac{1}{z}\Li(\xi_iv;z)\frac{\log^{r-1}z}{(r-1)!}
        -\frac{1}{z}\Li(\xi_iv \sh \xi_0^{r-1};z) \\
        & = \frac{d}{dz}\Li(\xi_iv;z)\cdot\frac{\log^rz}{r!}.
\end{align*}
Since $\xi_iv \in S^0$, we have already proved \eqref{eq:diffeq_Limap1} for
$\ds \frac{d}{dz}\Li(\xi_iv;z)$. Hence we have 
\begin{align*}
        \frac{d}{dz}\Li(\xi_iv;z)\cdot\frac{\log^rz}{r!} 
              = \begin{cases}
             \dfrac{1}{z}\Li(v;z)\dfrac{\log^rz}{r!} = \dfrac{1}{z}\Li(v \sh \xi_0^r;z)
                                               & \quad (i=0), \\
             {} \\
             \dfrac{1}{1-z}\Li(v;z)\dfrac{\log^rz}{r!}=\dfrac{1}{1-z}\Li(v \sh \xi_0^r;z) 
                                               & \quad (i=1).
            \end{cases}
\end{align*}
This proves \eqref{eq:diffeq_Limap1} for $u=v \sh \xi_0^r$. 
\end{proof}

In \cite{IKZ}, the following proposition is implicitly written:
\begin{prop}[IKZ]
We put 
\begin{align}\label{eq:IKZ0}
         \varPhi=\sum_{w \in S} w \otimes W \in S \, \widehat{\otimes} \, \cU
\end{align}
where $\widehat{\otimes}$ denotes the topological tensor product. Then $\varPhi$ decomposes as follows:
\begin{align}\label{eq:IKZ1}
      \varPhi=\Big(\sum_{w\in S}\reg^0(w)\otimes W\Big)\exp_{\sh}(\xi_0\otimes X_0)
\end{align}
where $\exp_{\sh}(\xi_0\otimes X_0)=\sum_{k=0}^\infty \xi_0^k\otimes X_0^k$. Furthermore we have 
\begin{align}\label{eq:IKZ2}
   \varPhi=\exp_{\sh}(\xi_1\otimes X_1)\Big(\sum_{w\in S}\reg^{10}(w)\otimes W\Big)
      \exp_{\sh}(\xi_0\otimes X_0).
\end{align}
\end{prop}

Since, by using \eqref{eq:IKZ0} and \eqref{eq:IKZ1}, we have 
\begin{align*}
   G(z) = \big(\Li \, \otimes \, \mathrm{id}_{\cU} \big)(\varPhi)
        = \Big(\sum_{w \in S}\Li(\reg^0(w);z)W \Big) z^{X_0},
\end{align*}
$G(z)$ satisfies the same asymptotic condition as $\cL(z)$. Hence
\begin{align*}
     G(z)=\cL(z).
\end{align*}

Moreover we obtain the following proposition.

\begin{prop}[\cite{OkU}]\label{prop:fundamental0}
The fundamental solution $\cL(z)$ normalized at $z=0$ can be written as follows:
\begin{align}
    \cL(z)&=\sum_{w \in S}\Li(w;z)W \label{eq:fundamental01}\\
          &=\Big(\sum_{w \in S}\Li(\reg^0(w);z)W \Big)z^{X_0} \label{eq:fundamental02}\\
          &=(1-z)^{-X_1}\Big(\sum_{w \in S}\Li(\reg^{10}(w);z)W \Big)z^{X_0}. \label{eq:fundamental03}
\end{align}
\end{prop}

We consider $\cL^{(1)}(z)$, the fundamental solution normalized at $z=1$ of \eqref{KZE1}. 
By the coordinate transformation $z \mapsto t=f(z)=1-z$, \eqref{KZE1} is transformed to 
\begin{align*}
         \frac{dG}{dt}=\Big(\frac{-X_1}{t}+\frac{-X_0}{1-t}\Big)G.
\end{align*}
Hence we can show that $\cL^{(1)}(z)$ which satisfies 
\begin{align*}
             \cL^{(1)}(z)=\hcL^{(1)}(z)(1-z)^{-X_1}
\end{align*}
where $\hcL^{(1)}(z)$ is holomorphic at $z=1$ and $\hcL^{(1)}(1)=\bunit$, uniquely exists. 
$\hcL^{(1)}(z)$ is expressed as follows:
\begin{align}\label{eq:MPLs1}
       & \hcL^{(1)}(z)=\sum_{s=0}^{\infty} \hcL^{(1)}_s(z), \nonumber \\
       & \hcL^{(1)}_s(z) = \!\!\!\sum_{k_1+\cdots+k_r=s}\!\!\!\Li_{k_1,\ldots,k_r}(1-z) \times \nonumber \\
       & \hspace{20mm} \times \ad(-X_1)^{k_1-1}\mu(-X_0)\cdots\ad(-X_1)^{k_r-1}\mu(-X_0)(\bunit).
\end{align}
From Proposition \ref{prop:fundamental0}, we obtain 
\begin{align}
         \cL^{(1)}(z)&=\sum_{w \in S}\Li(w;1-z)f_*(W) \label{eq:fundamental11}\\
         &=\Big(\sum_{w \in S}\Li(\reg^0(w);1-z)f_*(W)\Big)(1-z)^{-X_1} \label{eq:fundamental12}\\
         &=z^{X_0}\Big(\sum_{w \in S}\Li(\reg^{10}(w);1-z)f_*(W)\Big)(1-z)^{-X_1}, \label{eq:fundamental13}
\end{align}
where $f_*: \cU \to \cU$ is an automorphism defined by $f_*(X_0)=\!-X_1$, $f_*(X_1)=\!-X_0$.

\subsection{Drinfel'd associator and the connection problem of KZE1}

We set
\begin{align}
        \varPhi_{\rm KZ}=\varPhi_{\rm KZ}(X_0,X_1)=\sum_{w\in S}\zeta(\reg^{10}(w))W
\end{align}
which is referred to as the Drinfel'd associator. This is also a grouplike element of 
$\tcU$.

\begin{prop}[\cite{OkU}]\label{prop:connection1}
We have
\begin{enumerate}
\item $\cL(z)=\cL^{(1)}(z)\varPhi_{\rm KZ}$.
\item The relation \ $(\cL^{(1)}(z))^{-1}\cL(z)=\varPhi_{\rm KZ}$ \ is equivalent to 
\begin{align}\label{eq:gif}
    \sum_{uv=w}\Li(\tau(u);1-z)\Li(v;z)=\zeta(\reg^{10}(w)) \quad \text{for $\forall w$ word in $S$}.
\end{align}
Here $\tau: S \to S$ is an anti-automorphism defined by $\tau(\xi_0)=\xi_1,\; \tau(\xi_1)=\xi_0$. 
\end{enumerate}
\end{prop}
\begin{proof}
\begin{enumerate}
\item Note that $\Li_{k_1,\dots,k_r}(z)$ is holomorphic in $D \!=\! \C \setminus \{z=x \,|\, 1 \!\le\! x\}$, 
and $\Li_{k_1,\dots,k_r}(1-z)$ is holomorphic in $D'=\C \setminus \{z=x \,|\, x \le 0\}$. 
From $\eqref{eq:MPLs0}$ and $\eqref{eq:MPLs1}$, $\big(\cL^{(1)}(z)\big)^{-1}\cL(z)$ is holomorphic in a 
simply connected domain $D \cap D'$. Since $\ds \frac{d}{dz}\left\{\big(\cL^{(1)}(z)\big)^{-1}\cL(z)\right\}=0$, 
$\big(\cL^{(1)}(z)\big)^{-1}\cL(z)=C$ is constant. 
From $\eqref{eq:fundamental03}, \eqref{eq:fundamental13}$, taking the limit $z \to 1, z \in D \cap D'$,
we have 
\begin{align*}
        C &= \lim_{z \to 1} \ 
           \Big\{ (1-z)^{X_1}\Big( \sum_{w \in S}\Li(\reg^{10}(w);1-z)f_*(W) \Big)^{-1}z^{-X_0} \\
        & \hspace{20mm} \cdot (1-z)^{-X_1}\Big( \sum_{w \in S}\Li(\reg^{10}(w);z)W \Big)z^{X_0} \Big\} \\
          &= \sum_{w \in S} \zeta(\reg^{10}(w))W =\varPhi_{\rm KZ}
\end{align*}
\item \ Since $\cL^{(1)}(z)$ is a grouplike element of $\tcU$, from \eqref{eq:fundamental11}, 
we have 
\begin{align*}
      \big(\cL^{(1)}(z)\big)^{-1}=\sum_{w \in S}\Li(w;1-z)(\rho \circ f_*)(W)
\end{align*}
where $\rho:\cU \to \cU$ is an anti-homomorphism defined by $\rho(X_0)=-X_0$, $\rho(X_1)=-X_1$,
which is the antipode of $\cU$ as a Hopf algebra. Hence, introducing an anti-homomorphism 
$\tau: S \to S$ defined by $\tau(\xi_0)=\xi_1, \ \tau(\xi_1)=\xi_0$, we have 
\begin{align*}
        \big(\cL^{(1)}(z)\big)^{-1}=\sum_{w \in S}\Li(\tau(w);1-z) W.
\end{align*}
Therefore we have
\begin{align*}
        \big(\cL^{(1)}(z)\big)^{-1}\cL(z) & =\Big( \sum_{u \in S}\Li(\tau(u);1-z)U \Big) \cdot
                                  \Big(\sum_{v \in S}\Li(v;1-z)V \Big) \\
                                  &= \sum_{w}\Big(\sum_{uv=w}\Li(\tau(u);1-z)\Li(v;z)\Big)W.
\end{align*}
This proves \eqref{eq:gif}.
\end{enumerate}
\end{proof}

The relation \eqref{eq:gif} is called the generalized inversion formula.

\begin{prop}[\cite{OkU}]\label{prop:duality}
We have
\begin{enumerate}
\item $\varPhi_{\rm KZ}(X_0,X_1)\varPhi_{\rm KZ}(-X_1,-X_0)=\bunit$ \quad (the duality relation of 
      $\varPhi_{\rm KZ}$). \label{prop:duality:1}
\item From \eqref{prop:duality:1}, we obtain the duality relation of MZV:
\begin{align}
       \zeta(\reg^{10}(w))=\zeta(\reg^{10}(\tau(w))) \quad \text{for} \ \forall w \in S.
\end{align}
\end{enumerate}
\end{prop}
\begin{proof}
Taking the limit $z \to 0, z \in D \cap D'$ of $(\cL^{(1)}(z))^{-1}\cL(z)$, we have
\begin{align*}
       \varPhi_{\rm KZ}(X_0,X_1)= \Big( \sum_{w} \Li(\reg^{10}(w);1)f_*(W)\Big)^{-1}
       =\varPhi_{\rm KZ}(-X_1,-X_0)^{-1}.
\end{align*}
Since the Drinfel'd associator is grouplike, we have 
\begin{align*}
      \varPhi_{\rm KZ}(-X_1,-X_0)^{-1} &= \sum_{w}\zeta(\reg^{10}(w))(\iota \circ f_*)(W) \\
                                    &= \sum_{w}\zeta(\reg^{10}(\tau(w))W.
\end{align*}
This completes the proof.
\end{proof}

For examples of Proposition \ref{prop:connection1} and Proposition \ref{prop:duality}, we present the 
following:
\begin{enumerate}
\item $\Li_2(z)-\log z \Li_1(z)+\Li_2(1-z)=\zeta(2)$. \ (Euler's inversion formula \cite{Le})
\item $\zeta(3)=\zeta(2,1)$. \ 
      (Indeed, $\zeta(3)=\zeta(\xi^2_0\xi_1)=\zeta(\tau(\xi^2_0\xi_1))=\zeta(\xi_0\xi^2_1)=\zeta(2,1)$.)
\end{enumerate}

\subsection{Schlesinger type equation of one variable}

Consideration on \eqref{KZE1} generalizes to the case of the Schlesinger type equation.
\begin{align}\tag{SE1}\label{SE1}
        \frac{dG}{dz}=\Big(\frac{X_0}{z}+\sum_{j=1}^m\frac{a_jX_j}{1-a_jz}\Big)G
\end{align}
where $a_1,\ldots,a_m \in \C\setminus\{0\}$ are distinct points. 
The coefficients $X_0,X_1,\ldots,X_m$ are generators of a free Lie algebra
\begin{align*}
                  \fX=\C\{X_0,X_1,\ldots,X_m\}.
\end{align*}
The universal enveloping algebra $\cU=\cU(\fX)=\C\langle X_0,X_1,\ldots,X_m\rangle$ is 
a non-commutative polynomial ring and 
a completion $\widetilde{\cU}=\widetilde{\cU}(\fX)=\C\langle\langle X_0,X_1,\ldots,X_m\rangle\rangle$ 
is a non-commutative formal power series ring.

The one forms $\xi_0,\xi_1,\ldots,\xi_m$ are defined by
\begin{align*}
       \xi_0=\frac{dz}{z},\qquad \xi_j=\frac{a_jdz}{1-a_jz}\quad (1\le j \le m)
\end{align*}
and let $S=S(\xi_0,\xi_1,\ldots,\xi_m)=(\C\langle \xi_0,\xi_1,\ldots,\xi_m\rangle,\sh,\bnull)$ be 
a shuffle algebra generated by $\xi_0,\xi_1,\ldots,\xi_m$.

\begin{prop}\label{prop:SE1}
The Schlesinger type equation \eqref{SE1} has a unique solution $\cL(z)$ satisfying the asymptotic condition 
\begin{align*}
        & \cL(z)=\hcL(z)z^{X_0}, \\
        & \hcL(z) \ \text{is holomorphic at $z=0$ and $\hcL(0)=\bunit$}.
\end{align*}
\end{prop}
\begin{proof}
Similarly as in the case of \eqref{KZE1}, 
$\ds \hcL(z)=\sum_{s=0}^\infty \hcL_s(z) \ (\hcL_s(z) \in \cU_s)$ are expressed as 
\begin{align*}
  \hcL_s(z) &=\int_0^z \Big(\xi_0\otimes \ad(X_0)+\sum_{j=1}^m\xi_j\otimes\mu(X_j)\Big)^s(\bnull\otimes\bunit)\\
            &=\!\!\!\!\!\!\sum_{\substack{k_1+\cdots+k_r=s\\i_1,\ldots,i_r\in\{1,\ldots,m\}}}
            \!\!\!\!\!\!\!\!\!L({}^{k_1}a_{i_1}\cdots{}^{k_r}a_{i_r};z)
            \ad(X_0)^{k_1-1}\mu(X_{i_1})\cdots\ad(X_0)^{k_r-1}\mu(X_{i_r})(\bunit).
\end{align*}
Here $L({}^{k_1}a_{i_1}\cdots{}^{k_r}a_{i_r};z)$ stands for a hyperlogarithm which is by definition 
\begin{align*}
&L({}^{k_1}a_{i_1}\cdots{}^{k_r}a_{i_r};z):=\int_0^z\xi_0^{k_1-1}\xi_{i_1}\cdots\xi_0^{k_r-1}\xi_{i_r}\\
&\hspace{3cm}= \sum_{n_1>n_2\cdots>n_r>0}\frac{a_{i_1}^{n_1-n_2}a_{i_2}^{n_2-n_3}
\cdots a_{i_r}^{n_r}}{n_1^{k_1}\cdots n_r^{k_r}}z^{n_1}.
\end{align*}
This Taylor expansion is absolutely convergent for $\ds |z|<\min_{1\le k \le r}|a_{i_k}|^{-1}$.
\end{proof}

The solution $\cL(z)$ in Proposition \ref{prop:SE1} is referred to as the fundamental solution 
normalized at $z=0$ of \eqref{SE1}. $\cL(z)$, $\hcL(z)$ are grouplike elements of $\tcU$.

\section{KZE of two variables}
\subsection{Definition of KZE of two variables}

By the cubic coordinate on $\cM_{0,5}: z_1=r(1,4;3,5), z_2=r(2,1;3,5)$, 
KZE on $\cM_{0,5}$ is represented as the KZ equation of two variables (KZE2):
\begin{gather}\tag{KZE2}\label{KZE2}
dG = \varOmega G, \quad \varOmega=\xi_1 X_1 + \xi_{11} X_{11} + \xi_2 X_2 + \xi_{22} X_{22} + \xi_{12} X_{12}, \\[1ex]
\xi_1=\frac{dz_1}{z_1},\;\; \xi_{11}=\frac{dz_1}{1-z_1},\;\; \xi_2=\frac{dz_2}{z_2},\;\; \xi_{22}=\frac{dz_2}{1-z_2},\;\; \xi_{12}=\frac{d(z_1z_2)}{1-z_1z_2}. \notag \\[-6ex] \notag
\end{gather}

\vspace{3mm}

\noindent
The singular divisors $D$ of \eqref{KZE2} are 
\begin{align*}
             D=\{z_1=0,\, 1, \, \infty\} \cup \{z_2=0,\, 1, \, \infty\} \cup \{z_1z_2=1\}. 
\end{align*}
The coefficients $X_i,X_{ii},X_{12}$ are given by
\begin{align*}
X_1&=\varOmega_{12}+\varOmega_{13}+\varOmega_{14},\quad X_2=\varOmega_{23},\\
X_{11}&=-\varOmega_{14},\quad X_{22}=-\varOmega_{12},\quad X_{12}=-\varOmega_{24}, 
\end{align*}
which satisfy
\begin{align}\label{IR2}
\begin{cases}
[X_1,X_2]=[X_{11},X_2]=[X_1,X_{22}]=0, \\
[X_{11},X_{22}]=[-X_{11},X_{12}]=[X_{22},X_{12}]=[-X_1+X_2,X_{12}].
\end{cases}
\end{align}
$\fX$ is the Lie algebra generated by $X_1,\ldots,X_{12}$ with \eqref{IR2} as the fundamental relations.

Note that \eqref{AR} for this case is
\begin{align}
\begin{cases}
\xi_1\wedge\xi_{11}=0,\quad \xi_2\wedge\xi_{22}=0,\\
(\xi_1+\xi_2)\wedge\xi_{12}=0,\\
\xi_{11}\wedge\xi_{12}+\xi_{22}\wedge(\xi_{11}-\xi_{12})-\xi_2\wedge\xi_{12}=0.
\end{cases}
\end{align}

\vspace{5mm}

\hspace{-25mm}
\begin{picture}(0,0)(0,4)
\put(3,-0.3){\scalebox{0.5}{\includegraphics{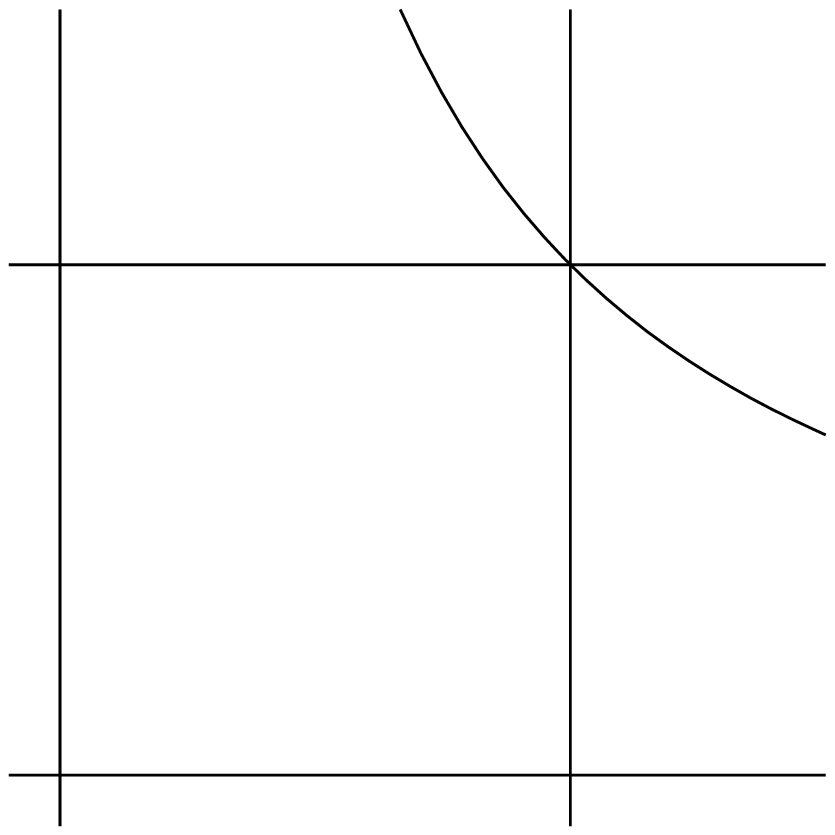}}}
\put(2.4,-0.3){\small $(0,0)$}
\put(6.1,-0.3){\small $(1,0)$}
\put(2.4,2.9){\small $(0,1)$}
\put(6.1,2.9){\small $(1,1)$}
\put(7.5,0.0){\small $z_1$}
\put(3.2,4.2){\small $z_2$}
\put(9.0,2){\shortstack[l]{singular divisors\\$D=\{z_1=0,1,\infty\}$\\\hspace{1.5cm}$\cup\{z_2=0,1,\infty\}$\\\hspace{2.5cm}$\cup\{z_1z_2=1\}$}}
\end{picture}

\vspace{45mm}

\subsection{Fiber space structure of $\cM_{0,5}$ and decomposition of $\fX$}

Consider the projection 
\begin{align}
           p_2: \cM_{0,5} \to \cM_{0,4}: \ 
            [x_1,x_2,x_3,x_4,x_5] \mapsto [x_1,x_3,x_4,x_5].
\end{align}

\vspace{10mm}

\hspace{-65mm}
\begin{picture}(0,4)(-8.5,-1.5)
\put(1,-1.2){$\cM_{0,4}\cong \bP^1\setminus\{0,1,\infty\}$}
\put(0,0){\line(1,0){5}}
\put(0.5,0){\circle{0.1}}
\put(0.4,-0.5){0}
\put(2.0,0){\circle{0.1}}
\put(1.9,-0.5){1}
\put(3.5,0){\circle{0.1}}
\put(3.3,-0.5){$\infty$}
\put(1.0,2.5){$\cM_{0,5}$}
\put(1.6,2.3){\vector(0,-1){1}}
\put(1.8,1.7){$p_2$}
\put(2.65,0){\circle*{0.1}}
\put(2.55,-0.5){$z_1$}
\put(2.65,0.3){\line(0,1){2.5}}
\put(2.75,1.7){\shortstack[l]{$p_2^{-1}(z_1)$\\\quad $\cong \bP^1\setminus\{0,1,\infty,z_1^{-1}\}$}}
\end{picture}

\noindent
From this, we obtain the semi-direct product of the fundamental groups;
\begin{align*}
          \pi_1(\cM_{0,5},(z_1,z_2)) \cong  \pi_1(\bP^1\setminus\{0,1,\infty,z_1^{-1}\},z_2) 
                                                 \rtimes \pi_1(\cM_{0,4},z_1).
\end{align*}
Taking Lie algebras, we have
\begin{align*}
       \fX = \C\{X_2,X_{22},X_{12}\}\oplus \C\{X_1,X_{11}\}
\end{align*}
($\C\{X_2,X_{22},X_{12}\}$ is a Lie ideal). We have similar decomposition
\begin{align*}
\fX = \C\{X_1,X_{11},X_{12}\}\oplus \C\{X_2,X_{22}\}
\end{align*}
($\C\{X_1,X_{11},X_{12}\}$ is a Lie ideal).

\begin{prop}\label{prop:decompo_X}
The following decomposition holds:
\begin{align}
   & \fX= \C\{X_1,X_{11},X_{12}\}\oplus \C\{X_2,X_{22}\} = \C\{X_2,X_{22},X_{12}\}\oplus \C\{X_1,X_{11}\}, \\
   & \hspace{20mm}  \cU(\fX)= \cU(\C\{X_1,X_{11},X_{12}\})\otimes \cU(\C\{X_2,X_{22}\}) \nonumber \\
   & \hspace{28mm}          = \cU(\C\{X_2,X_{22},X_{12}\})\otimes \cU(\C\{X_1,X_{11}\}) 
\end{align}
Moreover, $\C\{X_1,X_{11},X_{12}\}$, $\C\{X_2,X_{22},X_{12}\}$ are Lie ideals of $\fX$.
\end{prop}

\subsection{Reduced bar algebra}

Now let
\begin{equation*}
S=S(\xi_1,\xi_{11},\xi_2,\xi_{22},\xi_{12})=(\C\langle\xi_1,\xi_{11},\xi_2,\xi_{22},\xi_{12}\rangle,\sh,\bnull)
\end{equation*}
be a shuffle algebra generated by $\xi_1,\dots,\xi_{12}$. We say that an element
\begin{equation*}
       \varphi=\sum_{I=\{i_1,\ldots,i_s\}} c_i \omega_{i_1}\cdots\omega_{i_s}
\end{equation*}
($\omega_{i_k} \in \{\xi_1,\ldots,\xi_{12}\}$) 
satisfies Chen's integrability condition (CIC) if and only if, for $1\le\forall l \le s-1$, 
\begin{align}\tag{CIC}\label{Chen's_integrability}
   \sum_I c_I \; \omega_{i_1}\otimes\cdots\otimes\omega_{i_l}\wedge\omega_{i_{l+1}}\otimes\cdots
               \otimes\omega_{i_s}=0
\end{align}
holds as a multiple differential form \cite{C2}.
Let $\cB$ be a subalgebra of $S$ of elements satisfying \eqref{Chen's_integrability}: 
\begin{align}
        \cB=\{\varphi \in S\;|\; \varphi \text{ satisfies \eqref{Chen's_integrability}}\}.
\end{align}
We refer to $\cB$ as the reduced bar algebra.

For $\varphi \in \cB$, the iterated integral $\ds \int_{(a_1,a_2)}^{(z_1,z_2)} \varphi$ 
depends only on the homotopy class of the integral contour and defines 
a many-valued analytic function on $\bP^1 \times \bP^1 \setminus D$.

As Hopf algebras, $\cU$ and $\cB$ are dual to each other.

\begin{prop}[\cite{OU}]\label{prop:bar}
There exist isomorphisms of shuffle algebras
\begin{align}
   \iota_{1\otimes2}&: \cB \overset{\cong}{\longrightarrow} 
  S(\xi_1,\xi_{11},\widetilde{\xi}^{(1)}_{12})\otimes S(\xi_2,\xi_{22}),\\
   \iota_{2\otimes1}&: \cB \overset{\cong}{\longrightarrow} 
  S(\xi_2,\xi_{22},\widetilde{\xi}^{(2)}_{12})\otimes S(\xi_1,\xi_{11})
\end{align}
where $\ds \widetilde{\xi}^{(1)}_{12}=\frac{z_2dz_1}{1-z_1z_2}, \ 
\widetilde{\xi}^{(2)}_{12}=\frac{z_1dz_2}{1-z_1z_2} \quad (\widetilde{\xi}^{(1)}_{12}+
\widetilde{\xi}^{(2)}_{12}=\xi_{12})$. 
The isomorphism $\iota_{1\otimes2}$ is defined through the following procedure:\\
$(i)$ \  For $\varphi \in \cB$, pick up the terms 
         $\varphi_1\varphi_2 \in S(\xi_1,\xi_{11},\xi_{12})\cdot S(\xi_2,\xi_{22})$.\\
$(ii)$ \  change each term $\varphi_1\varphi_2$ to 
         $\varphi_1 \otimes \varphi_2 \in S(\xi_1,\xi_{11},\xi_{12})\otimes S(\xi_2,\xi_{22})$.\\
$(iii)$ \  replace $\xi_{12}$ to $\widetilde{\xi}_{12}^{(1)}$.\\
$\iota_{2\otimes1}$ is similarly defined. 
\end{prop}

Let $\cB^0$ be a subalgebra of elements ending with other than $\xi_1,\xi_2$: 
\begin{equation*}
\cB^0=\{\varphi \in \cB\;|\; \varphi=\sum c_I \omega_{i_1}\cdots\omega_{i_s},\; \omega_{i_s}\neq \xi_1,\xi_2\}.
\end{equation*}
For $\varphi \in \cB^0$, we can put an iterated integral 
$\ds \int_{(0,0)}^{(z_1,z_2)} \varphi$.\\

\begin{prop}[\cite{OU}]\label{prop:bar0}
We have isomorphisms of shuffle algebras
\begin{align}
\iota_{1\otimes2}&: \cB^0 \overset{\cong}{\longrightarrow} S^0(\xi_1,\xi_{11},\widetilde{\xi}^{(1)}_{12})\otimes S^0(\xi_2,\xi_{22}),\\
\iota_{2\otimes1}&: \cB^0 \overset{\cong}{\longrightarrow} S^0(\xi_2,\xi_{22},\widetilde{\xi}^{(2)}_{12})\otimes S^0(\xi_1,\xi_{11}),
\end{align}
where $S^0(A)=\{\varphi \in S(A)\;|\; \varphi=\sum c_I \omega_{i_1}\cdots\omega_{i_s},\; \omega_{i_s}\neq \xi_1,\xi_2\}$.
\end{prop}

\subsection{Fundamental solution of KZE2}

Let $\varOmega_0=\xi_1 X_1+\xi_2 X_2$ (the singular part of $\varOmega$ at $(0,0)$) 
and $\varOmega'=\varOmega-\varOmega_0$ (the regular part of $\varOmega$ at $(0,0)$).

\begin{prop}\label{prop:KZE2}
\ref{KZE2} has a unique solution $\cL(z_1,z_2)$ satisfying the asymptotic condition: 
\begin{align*}
    & \cL(z_1,z_2)=\hcL(z_1,z_2)z_1^{X_1}z_2^{X_2}, \\
    & \hcL(z_1,z_2) \ \text{is holomorphic at $(0,0)$ and $\hcL(0,0)=\bunit$}.
\end{align*}
\end{prop}
\begin{proof}
Since the fundamental relations \eqref{IR2} of $\fX$ is homogeneous, $\cU$ has the canonical gradation: 
$\ds \cU= \bigoplus_{s=0}^{\infty} \cU_s$. Similarly, the reduced bar algebra $\cB$ and the subalgebra
$\cB^0$ are graded: $\ds \cB= \bigoplus_{s=0}^{\infty} \cB_s, \ 
\cB^0= \bigoplus_{s=0}^{\infty} \cB_s^0 \quad (B_s^0=B_s \cap B^0)$.

Let $\ds \hcL(z_1,z_2)=\sum_{s=0}^\infty \hcL_s(z_1,z_2), \ (\hcL_s \in \cU_s)$. 
Then we have the following recursive relation:
\begin{align*}
         d\hcL_{s+1}(z_1,z_2)&=[\varOmega_0,\hcL_s]+\varOmega'\hcL_s.
\end{align*}

\begin{lem}[\cite{OU}, Proposition 17]
We have 
\begin{align*}
  \left(\ad(\varOmega_0)+\mu(\varOmega')\right)^s(\bnull \otimes \bunit) \in \cB^0_s\otimes \cU_s
\end{align*}
where we set $\ad(\omega\otimes X)(\varphi\otimes A)=\omega\varphi \otimes [X,A]$, 
$\mu(\omega\otimes X)(\varphi\otimes A)=\omega\varphi \otimes XA$.
\end{lem}

Hence we have $\hcL_0(z)=\bunit$, and for $s \ge 1$
\begin{align}\label{eq:Ls}
     \hcL_s(z_1,z_2)=\int_{(0,0)}^{(z_1,z_2)} 
     \left(\ad(\varOmega_0)+\mu(\varOmega')\right)^s(\bnull \otimes \bunit).
\end{align}
This shows that the fundamental solution $\cL(z_1,z_2)$ exists and is unique.
\end{proof}

The solution $\cL(z_1,z_2)$ is called the fundamental solution normalized at $(0,0)$ of \eqref{KZE2} .
We should observe that $\cL(z_1,z_2)$ and $\hcL(z_1,z_2)$ are grouplike elements of $\widetilde{\cU}$ 
which is the completion of $\cU$ with respect to the gradation, and that $\hcL(z_1,z_2)$ is 
single valued and holomorphic in $\{(z_1,z_2) \,|\, |z_1|<1, \ |z_2|<1\}$.

\subsection{Decomposition theorem of fundamental solution}

\begin{prop}\label{prop:decomposition}
The fundamental solution $\cL(z_1,z_2)$ normalized at $(0,0)$ decomposes as
\begin{align}\tag{$\sharp$}\label{eq:decomposition}
     \cL(z_1,z_2) = \cL_{2 \otimes 1}^{(2)}(z_1,z_2)\cL_{2 \otimes 1}^{(1)}(z_1)
                  =\cL_{1 \otimes 2}^{(1)}(z_1,z_2)\cL_{1 \otimes 2}^{(2)}(z_2).
\end{align}
Here 
$\cL^{(j)}_{i\otimes j}(z_j)$ is the fundamental solution normalized at $z_j=0$ of
\begin{gather}
\frac{dG}{dz_j}=\left(\frac{X_j}{z_j}+\frac{X_{jj}}{1-z_j}\right)G\tag{$\text{KZE1}_{i\otimes j}$}\\
\cL^{(j)}_{i\otimes j}(z_j)=\hcL^{(j)}_{i\otimes j}(z_j)z_j^{X_j},\quad \hcL^{(j)}_{i\otimes j}(0)=\bunit. \notag
\end{gather}
$\cL^{(i)}_{i\otimes j}(z_1,z_2)$ is the fundamental solution normalized at $z_i=0$ of
\begin{gather}
\frac{dG}{dz_i}=\left(\frac{X_i}{z_i}+\frac{X_{ii}}{1-z_i}+\frac{z_jX_{12}}{1-z_1z_2}\right)G\tag{$\text{SE1}_{i\otimes j}$}\\
\cL^{(i)}_{i\otimes j}(z_1,z_2)=\hcL^{(i)}_{i\otimes j}(z_1,z_2)z_i^{X_i},\quad \hcL^{(i)}_{i\otimes j}\Big|_{z_i=0}=\bunit. \notag
\end{gather}

\noindent The holomorphic part $\hcL(z_1,z_2)$ has similar decomposition as follows:
\begin{align}\tag{$\sharp'$}\label{eq:decomposition'}
         \hcL(z_1,z_2)=\hcL_{2 \otimes 1}^{(2)}(z_1,z_2)\hcL_{2 \otimes 1}^{(1)}(z_1)
                      =\hcL_{1 \otimes 2}^{(1)}(z_1,z_2)\hcL_{1 \otimes 2}^{(2)}(z_2).
\end{align}
\end{prop}
\begin{proof}
From the asymptotic condition of $\ds \cL(z_1,z_2)=\hcL(z_1,z_2)z_1^{X_1}z_2^{X_2}$, it follows that
$\ds G(z_1)=\hcL(z_1,0)z_1^{X_1}$ is the fundamental solution of ($\text{KZE1}_{2\otimes 1}$) . 
Hence $G(z_1)=\cL_{2 \otimes 1}^{(1)}(z_1)$. Put
\begin{align*}
        H(z_1,z_2)=\cL(z_1,z_2)\big(\cL_{2 \otimes 1}^{(1)}(z_1)\big)^{-1}.
\end{align*}
Then $H(z_1,z_2)$ is a solution of ($\text{SE1}_{2 \otimes 1}$).  
Since $[X_2,X_1]=[X_2,X_{11}]=0$, it has the asymptotic condition as follows:
\begin{align*}
      & H(z_1,z_2)=\hat{H}(z_1,z_2)z_2^{X_2}, \\
      & \widehat{H}(z_1,z_2)=\hcL(z_1,z_2)\big(\hcL(z_1,0)\big)^{-1}, \quad 
        \widehat{H}(z_1,0)=\bunit.
\end{align*}
Therefore $H(z_1,z_2)=\cL_{2 \otimes 1}^{(2)}(z_1,z_2)$. 
\end{proof}

We call this proposition ``Decomposition theorem for the fundamental solution''.

\section{Connection Problem of KZE on $\cM_{0,5}$}
\subsection{Compactification of $\cM_{0,5}$}

A smooth compactification of $\cM_{0,5}$ is given as follows \cite{Y}:
\begin{gather}
\overline{\cM}_{0,5}=\PGL(2, \C)\Big\backslash \big((\bP^1)^5\setminus \Delta'\big), \\
\Delta'=\{(x_1,x_2,x_3,x_4,x_5) \in (\bP^1)^5\;|\; \text{at least $3$ points coincide}\}. \nonumber
\end{gather}
Then
\begin{gather}
  \overline{\cM}_{0,5}\setminus\cM_{0,5}=\bigcup_{1\le i<j\le 5}D_{ij},\\
   D_{ij}=\{[x_1,x_2,x_3,x_4,x_5]\;|\; x_i=x_j\}.
\end{gather}
The divisors $D_{ij}$'s are rational curves in $\overline{\cM}_{0,5}$ and satisfy
\begin{align*}
   D_{ij}\cap D_{kl}=
     \begin{cases}
          \emptyset & (\{i,j\}\cap\{k,l\}\neq \emptyset,\; \{i,j\}\neq\{k,l\}),\\
          \{\text{one point}\}& (\{i,j\}\cap\{k,l\}=\emptyset).
     \end{cases}
\end{align*}

Let $\overline{\cM}_{0,5}(\R)$ be the real points of $\overline{\cM}_{0,5}$:
\begin{align*}
    \cM_{0,5}(\R)=\overline{\cM}_{0,5}(\R)\setminus\bigcup_{1\le i<j\le 5}D_{ij}(\R).
\end{align*}
Then $\cM_{0,5}(\R)$ has 12 connected components, and each component is a pentagon surrounded by 
five curves $D_{ij}(\R)$: 

\vspace{15mm}
\hspace{1mm}
\begin{picture}(0,8)(0,0)
\footnotesize
\put(0,0){\includegraphics{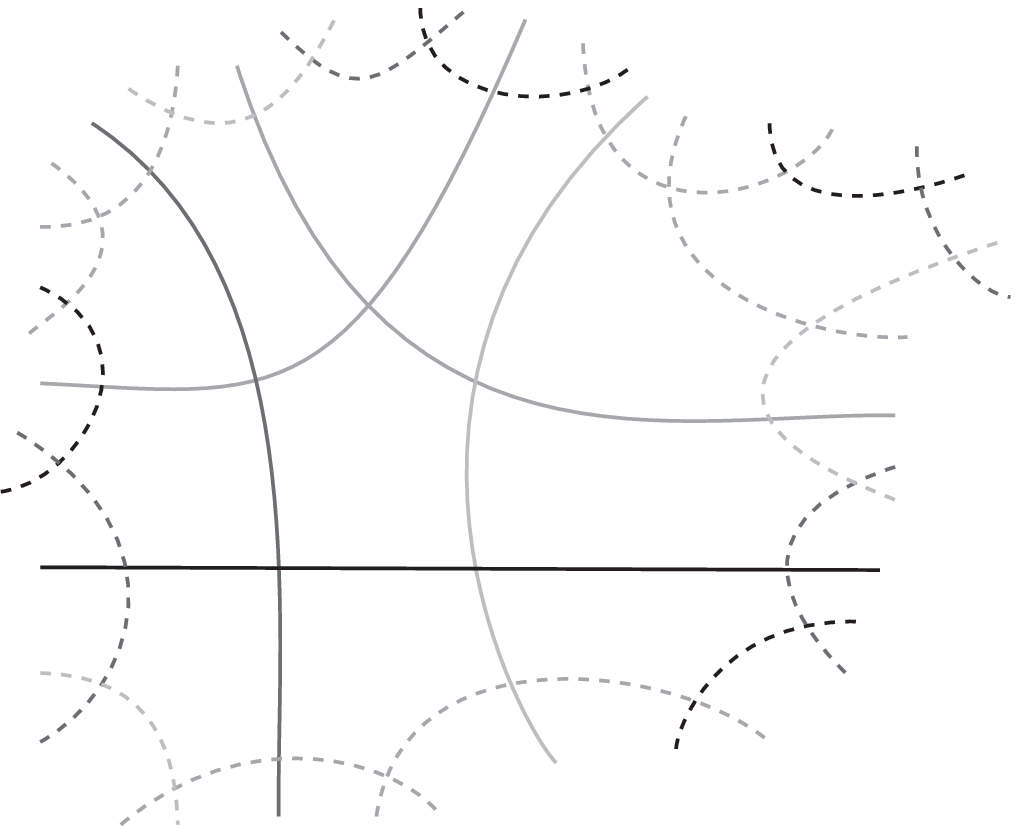}}
\put(9.1,2.5){$D_{23}$}
\put(2.6,-0.2){$D_{45}$}
\put(9.3,4.1){$D_{35}$}
\put(5.6,0.3){$D_{14}$}
\put(-0.2,4.5){$D_{12}$}
\put(-0.6,3.4){$D_{34}$}
\put(-0.1,0.7){$D_{15}$}
\put(0.9,-0.2){$D_{13}$}
\put(-0.2,1.6){$D_{24}$}
\put(8,0.7){$D_{25}$}
\put(3.6,3.8){$1$}
\put(3.6,1.6){$2$}
\put(2,3.6){$3$}
\put(2,1.6){$4$}
\put(6,3.5){$5$}
\put(6,1.9){$6$}
\put(1.5,5.4){$7$}
\put(2.5,5.8){$8$}
\put(6.2,5.5){$9$}
\put(4.6,5.8){$10$}
\put(3.5,6.7){$11$}
\put(8,5.8){$12$}
\put(8.7,2.2){$3$}
\put(8.7,3){$4$}
\put(0.6,2.1){$5$}
\put(0.6,2.9){$6$}
\put(3.2,0.2){$7$}
\put(2.2,0.2){$8$}
\put(4.6,0.6){$9$}
\put(5.9,0.9){$10$}
\put(8.5,3.8){$11$}
\put(4,0){$12$}
\put(8.5,4.5){$8$}
\put(7.2,0.7){$7$}
\put(7.7,1.3){$12$}
\put(8.6,1.75){$11$}
\put(1.5,0){$9$}
\put(0.9,0.7){$12$}
\put(0.3,1.2){$11$}
\put(0,3.6){$12$}
\put(0.4,4.1){$11$}
\put(0.4,4.7){$10$}
\put(0.3,5.2){$6$}
\put(0.5,5.7){$12$}
\put(0.5,6.3){$9$}
\put(1,6.6){$2$}
\put(1.4,7){$4$}
\put(1.4,7.5){$12$}
\put(2.05,7.4){$9$}
\put(2.7,7.6){$5$}
\put(3.05,8.1){$4$}
\put(3.6,7.9){$12$}
\put(4.4,8.3){$6$}
\put(4.7,7.8){$3$}
\put(5.4,7.7){$7$}
\put(6.1,7.8){$12$}
\put(6.2,7.3){$6$}
\put(6.5,6.9){$2$}
\put(7.2,6.7){$7$}
\put(8,7){$10$}
\put(8.6,6.6){$6$}
\put(9.5,6.7){$3$}
\put(9.7,6.2){$11$}
\put(10.1,5.6){$5$}
\put(9.1,5.2){$4$}
\put(9.1,3.4 ){$12$}
\put(1.5,2.8){$\bullet$}
\put(1.5,2.3){$\star$}
\put(8.2,2.3){$\bullet$}
\put(8.2,2.8){$\star$}
\end{picture}

\vspace{20mm}

\noindent
In the above figure, pentagons with the same number should be identified.

\subsection{Cubic coordinate system associated with a pentagon}

Let us fix the pentagon 1:

\begin{picture}(0,4.5)(0,-0.5)
\footnotesize
\put(0,0){\scalebox{0.5}{\includegraphics{m05.eps}}}
\put(1.2,0.3){$D_{23}$}
\put(2.7,1.0){$D_{14}$}
\put(2.4,2.7){$D_{35}$}
\put(0.7,3.0){$D_{12}$}
\put(-0.1,1.4){$D_{45}$}
\put(0.7,0.7){$P_0$}
\put(2.1,0.6){$P_1$}
\put(2.6,2.0){$P_2$}
\put(1.6,2.8){$P_3$}
\put(0.5,2.1){$P_4$}
\thicklines
\put(4.3,2.2){blow down}
\put(4.1,2){\vector(1,0){2}}
\put(6.5,0){\scalebox{0.45}{\includegraphics{m05_cubic.eps}}}
\put(10.5,0.3){$z_1$}
\put(9.2,4){$z_2$}
\put(8.9,-0.1){$z_1\!=\!1$}
\put(5.9,2.6){$z_2\!=\!1$}
\put(10.5,1.8){$z_1z_2\!=\!1$}
\put(6.7,0.1){$0$}
\end{picture}

Each vertex of this pentagon is given as follows:
\begin{align*}
   P_0&=[0,1,1,\infty,\infty],\quad P_1=[0,\infty,\infty,1,0],\quad P_2=[\infty,0,1,\infty,1],\\
   P_3&=[1,1,\infty,0,\infty],\quad P_4=[\infty,\infty,0,1,1].
\end{align*}
The cubic coordinate system
\begin{align}
    (z_1,z_2) &=\big(r(1,4;3.5), \ r(2,1;3.5) \big) \nonumber \\
              &=\left( \frac{(x_1-x_3)(x_4-x_5)}{(x_1-x_5)(x_4-x_3)}, \frac{(x_2-x_3)(x_1-x_5)}{(x_2-x_5)(x_1-x_3)} \right)
\end{align}
is a local coordinate system around $P_0$, and via the map
\begin{align*}
   [x_1,x_2,x_3,x_4,x_5] \ \mapsto \ (z_1,z_2)
\end{align*}
the pentagon 1 blows down to the square $\{(z_1,z_2) \,|\, 0 \le z_1 \le 1, \ 0 \le z_2 \le 1 \}$
as follows:
\begin{align*}
      & (z_1,z_2)(P_0)=(0,0),\ (z_1,z_2)(D_{23})=\{z_2=0\}, \ (z_1,z_2)(D_{45})=\{z_1=0\}, \\
      & (z_1,z_2)(D_{14})=\{z_1=1\}, \ (z_1,z_2)(D_{12})=\{z_2=1\}, \ (z_1,z_2)(D_{35})=(1,1).
\end{align*}
We should observe that $\cL(z_1,z_2)$, the fundamental solution normalized at $(z_1,z_2)=(0,0)$ of
\eqref{KZE2}, is a fundamental solution of \eqref{KZE} which is
single valued and holomorphic in a certain domain containing the pentagon 1. 

\vspace{2mm}

Put $\sigma=\begin{pmatrix}1&2&3&4&5\\5&4&1&3&2\end{pmatrix}$. 
Define $\sigma: \overline{\cM}_{0,5} \to \overline{\cM}_{0,5}$ as follows: 
For a point $P=[a_1,a_2,a_3,a_4,a_5]$, we set 
\begin{align}
     & \sigma(P)=[a_{\sigma^{-1}(1)},a_{\sigma^{-1}(2)},a_{\sigma^{-1}(3)},a_{\sigma^{-1}(4)},a_{\sigma^{-1}(5)}]\\
     & \hspace{7mm}(=[a_3,a_5,a_4,a_2,a_1]).\nonumber 
\end{align}
Then $\sigma$ is an automorphism (attached to the pentagon 1) of $\overline{\cM}_{0,5}$ 
satisfying $\sigma^5={\rm id}$ and 
\begin{align}
   \sigma(P_\alpha)=P_{\alpha+1} \ ({\rm mod} \ 5),\quad \sigma(D_{ij})=D_{\sigma(i),\sigma(j)} \quad (\forall i,j). 
\end{align}
Let $\sigma^*(x_i)=x_{\sigma(i)} \ (i=1,2,3,4,5)$ 
which induces an automorphism of the rational function field of $\overline{\cM}_{0,5}$. Put $(z_1^{(0)},z_2^{(0)})=(z_1,z_2)$. 
Then
\begin{align*}
    \begin{cases}
     z_1^{(1)}=\sigma^*(z_1^{(0)})=r(5,3;1,2)=z^{(0)}_2,\\
     {} \\
     z_2^{(1)}=\sigma^*(z_2^{(0)})=r(4,5;1,2)=\dfrac{1-z^{(0)}_1}{1-z^{(0)}_1z^{(0)}_2}
    \end{cases}
\end{align*}
are local coordinates around $P_1$. In general, we set 
\begin{align}
   \begin{cases}
     z_1^{(\alpha+1)}=\sigma^*(z_1^{(\alpha)})=z^{(\alpha)}_2,\\
     {} \\
     z_2^{(\alpha+1)}=\sigma^*(z_2^{(\alpha)})=\dfrac{1-z^{(\alpha)}_1}{1-z^{(\alpha)}_1z^{(\alpha)}_2}.
   \end{cases}
\end{align}
Then $(z_1^{(\alpha)},z_2^{(\alpha)})$ is a local coordinate system around $P_{\alpha}$, and is called the 
cubic coordinate system attached to $P_{\alpha}$.

Define the pull-back of $\omega_{ij}$ by
\begin{align}
     \sigma^*(\omega_{ij})=d\log(\sigma^*(x_i)-\sigma^*(x_j)),
\end{align}
and let $\sigma_*: \fX \to \fX$ be an induced automorphism defined through 
\begin{align}
\sum_{1\le i<j\le 5}\omega_{ij}\sigma_*(\varOmega_{ij})=\sum_{1\le i<j\le 5}\sigma^*(\omega_{ij})\varOmega_{ij},
\end{align}
namely,
\begin{align}
     \sigma_*(\varOmega_{ij})=\varOmega_{\sigma^{-1}(i),\sigma^{-1}(j)}. 
\end{align}
Then we have
\begin{align}
\varOmega=(\sigma^*\otimes \sigma_*^{-1})(\varOmega)=\sum_{1\le i<j\le 5}
                               \sigma^*(\omega_{ij})\sigma_*^{-1}(\varOmega_{ij}).
\end{align}

We represent $\varOmega$ by the cubic coordinate system attached to $P_{\alpha}$: 
Let 
\begin{align*}
&X_i^{(0)}=X_i,\quad X_{ij}^{(0)}=X_{ij} \quad (1\le i\le 2,\; 1\le i\le j\le 2),\\
&X_i^{(\alpha+1)}=\sigma_*^{-1}(X_i^{(\alpha)}),\quad X_{ij}^{(\alpha+1)}=\sigma_*^{-1}(X_{ij}^{(\alpha)}) 
\quad  (\alpha=0,1,2,3).
\end{align*}
Then $\varOmega$ reads 
\begin{equation}
\varOmega=\xi^{(\alpha)}_1 X^{(\alpha)}_1 + \xi^{(\alpha)}_{11} X^{(\alpha)}_{11} + \xi^{(\alpha)}_2 X^{(\alpha)}_2 + \xi^{(\alpha)}_{22} X^{(\alpha)}_{22} + \xi^{(\alpha)}_{12} X^{(\alpha)}_{12},
\end{equation}
where
\begin{align*}
\xi^{(\alpha)}_i=\frac{dz^{(\alpha)}_i}{z^{(\alpha)}_i}, \quad 
\xi^{(\alpha)}_{ii}=\frac{dz^{(\alpha)}_i}{1-z^{(\alpha)}_i} \quad (i=1,2), \quad 
\xi^{(\alpha)}_{12}=\frac{d(z^{(\alpha)}_1z^{(\alpha)}_2)}{1-z^{(\alpha)}_1z^{(\alpha)}_2}.
\end{align*}

\noindent
Let $\cL^{(\alpha)}=\cL^{(\alpha)}(z_1^{(\alpha)},z_2^{(\alpha)})$ be the fundamental solution 
normalized at $P_{\alpha}$ of $\text{KZE2}_{\alpha}$
\begin{align}\tag{${\rm KZE2}_{\alpha}$}\label{eq:KEZ2a}
dG=\left(\xi^{(\alpha)}_1 X^{(\alpha)}_1 + \xi^{(\alpha)}_{11} X^{(\alpha)}_{11} + \xi^{(\alpha)}_2 X^{(\alpha)}_2 + \xi^{(\alpha)}_{22} X^{(\alpha)}_{22} + \xi^{(\alpha)}_{12} X^{(\alpha)}_{12}\right)G.
\end{align}
We should observe that each $\cL^{(\alpha)}$ is a fundamental solution of \eqref{KZE} which is single valued
and holomorphic in a certain domain containing the pentagon 1.

Let us find the connection matrix $C_{\alpha}$;
\begin{align}
        \cL^{(\alpha)}(z_1^{(\alpha)},z_2^{(\alpha)})=
        \cL^{(\alpha+1)}(z_1^{(\alpha+1)},z_2^{(\alpha+1)})C_{\alpha}.
\end{align}
To this equation, we substitute the decomposition \eqref{eq:decomposition} for $\cL^{(\alpha)},\cL^{(\alpha+1)}$
\begin{align*}
  \cL^{(\alpha)}(z_1^{(\alpha)},z_2^{(\alpha)})&=\cL^{(\alpha)(2)}_{2\otimes1}(z_1^{(\alpha)},z_2^{(\alpha)})
                                                 \cL^{(\alpha)(1)}_{2\otimes1}(z_1^{(\alpha)}),\\
  \cL^{(\alpha+1)}(z_1^{(\alpha+1)},z_2^{(\alpha+1)})&=\cL^{(\alpha+1)(1)}_{1\otimes2}(z_1^{(\alpha+1)},z_2^{(\alpha+1)})
                                                       \cL^{(\alpha+1)(2)}_{1\otimes2}(z_2^{(\alpha+1)}),
\end{align*}
where $\cL^{(\beta)(j)}_{i \otimes j}(z_j^{(\beta)}) \ (\beta=\alpha, \alpha+1)$ \ is the fundamental 
solution normalized at $z_j^{(\beta)}=0$ of $\text{KZE1}_{i\otimes j}^{(\beta)}$
\begin{gather*}\tag{$\text{KZE1}_{i\otimes j}^{(\beta)}$}
     \frac{dG}{dz_j^{(\beta)}}= \left(\frac{X_j^{(\beta)}}{z_j^{(\beta)}}+
                            \frac{X_{jj}^{(\beta)}}{1-z_j^{(\beta)}}\right) G   \\
\cL^{(\beta)(j)}_{i\otimes j}(z_j)=\hcL^{(\beta)(j)}_{i\otimes j}(z_j^{(\beta)})
               \big(z_j^{(\beta)}\big)^{X_j^{(\beta)}},\quad \hcL^{(\beta)(j)}_{i\otimes j}(0)=\bunit,
\end{gather*}
and $\cL^{(\beta)(i)}_{i \otimes j}(z_1^{(\beta)},z_2^{(\beta)})$ \ is the fundamental solution 
normalized at $z_i^{(\beta)}=0$ of $\text{SE1}_{i\otimes j}^{(\beta)}$
\begin{gather*}\tag{$\text{SE1}_{i\otimes j}^{(\beta)}$}
     \frac{dG}{dz_i^{(\beta)}}=\left(\frac{X_i^{(\beta)}}{z_i^{(\beta)}}
                           + \frac{X_{ii}^{(\beta)}}{1-z_i^{(\beta)}}
                           + \frac{z_j^{(\beta)}X_{12}^{(\beta)}}{1-z_1^{(\beta)}z_2^{(\beta)}}\right)G \\
\cL^{(\beta)(i)}_{i\otimes j}(z_1^{(\beta)},z_2^{(\beta)})
 =\hcL^{(\beta)(i)}_{i\otimes j}(z_1^{(\beta)},z_2^{(\beta)})\big(z_i^{(\beta)}\big)^{X_i^{(\beta)}},
       \quad \hcL^{(\beta)(i)}_{i\otimes j}\Big|_{z_i^{(\beta)}=0}=\bunit. \notag
\end{gather*}
Then we have
\begin{align}\label{eq:contiguity}
     \cL^{(\alpha)(2)}_{2\otimes1}(z_1^{(\alpha)})C_{\alpha}^{-1}
         &=\cL^{(\alpha)(2)}_{2\otimes1}(z_1^{(\alpha)},z_2^{(\alpha)})^{-1}
           \cL^{(\alpha+1)(1)}_{1\otimes2}(z_1^{(\alpha+1)},z_2^{(\alpha+1)}) \nonumber \\
    & \hspace{30mm} \cdot \cL^{(\alpha+1)(2)}_{1\otimes2}(z_2^{(\alpha+1)}).
\end{align}
Note that $\ds z_1^{(\alpha+1)}=z_2^{(\alpha)},\; z_2^{(\alpha+1)}=
\frac{1-z_1^{(\alpha)}}{1-z_1^{(\alpha)}z_2^{(\alpha)}}$ and 
$\ds X_1^{(\alpha+1)}=X_2^{(\alpha)},\; X_2^{(\alpha+1)}=-X_{11}^{(\alpha)}$. Since the left hand side of 
\eqref{eq:contiguity} is independent of the variable $z_2^{(\alpha)}$, first taking the limit of $z_2^{(\alpha)} \to 0$, 
we have 
\begin{align*} 
    \text{RHS} & = \cL^{(\alpha+1)(2)}_{1\otimes2}(1-z_1^{(\alpha)})\big(1-z_1^{(\alpha)}\big)
                           ^{-X_{11}^{(\alpha)}} \\
               & \sim \bunit\times (1-z_1^{(\alpha)})^{-X_{11}^{(\alpha)}} \quad (z_1^{(\alpha)} \to 1).
\end{align*}
Hence RHS is the fundamental solution normalized at $z_1^{(\alpha)}=0$ of 
$\text{KZE1}_{i\otimes j}^{(\alpha)}$. 
This implies $C_{\alpha}=\varPhi_{\rm KZ}(X_1^{(\alpha)},X_{11}^{(\alpha)})$.

\begin{thm} Put $\varPhi^{(\alpha)}_{\rm KZ}=\varPhi_{\rm KZ}(X_1^{(\alpha)},X_{11}^{(\alpha)})$.
\begin{enumerate}
\item We have
\begin{align}\label{thm:connection_pentagon}
        \cL^{(\alpha)}=\cL^{(\alpha+1)}\varPhi^{(\alpha)}_{\rm KZ} \quad (\alpha=0,1,2,3,4 \pmod{5}).
\end{align}
\item  As the compatibility condition for the connection relation \eqref{thm:connection_pentagon}, 
       we obtain the pentagon relation of the Drinfel'd associator
\begin{align}\label{thm:pentagon_relation}
     \varPhi^{(4)}_{\rm KZ}\;\varPhi^{(3)}_{\rm KZ}\;\varPhi^{(2)}_{\rm KZ}\;\varPhi^{(1)}_{\rm KZ}\;\varPhi^{(0)}_{\rm KZ}
      =\bunit.
\end{align} \label{thm:connection_pentagon:2}
\end{enumerate}
\end{thm}
\begin{proof}
We show \eqref{thm:connection_pentagon:2}. Since the connection relations \eqref{thm:connection_pentagon} hold in a certain domain 
containing the pentagon 1, we have 
\begin{align*}
     \cL^{(0)}=\cL^{(1)}\varPhi^{(0)}_{\rm KZ}=\cL^{(2)}\varPhi^{(1)}_{\rm KZ}\varPhi^{(0)}_{\rm KZ}
              =  \cdots = \cL^{(0)}\varPhi^{(4)}_{\rm KZ}\varPhi^{(3)}_{\rm KZ}
                          \varPhi^{(2)}_{\rm KZ}\varPhi^{(1)}_{\rm KZ}\varPhi^{(0)}_{\rm KZ}.
\end{align*}
From this we obtain the pentagon relation \eqref{thm:pentagon_relation}.
\end{proof}

\section{Five term relation of dilogarithms}
\subsection{Iterated integral representation along $C_{1 \otimes 2}, \ C_{2 \otimes 1}$ of 
           $\cL(z_1,z_2)$}

Let $C_{1 \otimes 2}, \ C_{2 \otimes 1}$ be contours defined in the following figure:

\hspace{20mm}
\begin{picture}(0,5)(0,0)
\footnotesize
\put(0,0){\scalebox{0.5}{\includegraphics{m05_cubic.eps}}}
\put(-0.4,0.1){$(0,0)$}
\put(3.2,0.1){$(1,0)$}
\put(-0.4,3.2){$(0,1)$}
\put(3.2,3.2){$(1,1)$}
\put(4.8,0.35){$z_1$}
\put(0.3,4.6){$z_2$}
\put(2.7,2.4){$(z_1,z_2)$}
\put(0.9,2.7){$C_{1\otimes2}$}
\put(2.8,1.7){$C_{2\otimes1}$}

\thicklines
\put(0.65,0.5){\vector(1,0){2}}
\put(2.65,0.5){\vector(0,1){2}}
\put(0.65,0.5){\vector(0,1){2}}
\put(0.65,2.5){\vector(1,0){2}}
\end{picture}

\vspace{5mm}

For 
$\varphi_2 \otimes \varphi_1 \in S^0(\xi_2,\xi_{22},\widetilde{\xi}_{12}^{(2)})\otimes S^0(\xi_1,\xi_{11})$, 
we set
\begin{align}
   \int_{2\otimes1}\varphi_2 \otimes \varphi_1 := \int_{C_{2\otimes1}}\varphi_2\varphi_1 
        = \int_{z_2=0}^{z_2}\!\!\varphi_2 \int_{z_1=0}^{z_1}\!\!\!\varphi_1.
\end{align}
Similarly for 
$\psi_1 \otimes \psi_2 \in S^0(\xi_1,\xi_{11},\widetilde{\xi}_{12}^{(1)})\otimes S^0(\xi_2,\xi_{22})$, 
we set
\begin{align}
  \int_{1\otimes2}\psi_1 \otimes \psi_2 := \int_{C_{1\otimes2}}\psi_1\psi_2
               =\int_{z_1=0}^{z_1}\!\!\psi_1 \int_{z_2=0}^{z_2}\!\!\!\psi_2.
\end{align}

\begin{prop}\label{prop:GHPR}
For $\varphi \in \cB^0$, we have 
\begin{align*}
       \int_{(0,0)}^{(z_1,z_2)}\varphi&=\int_{C_{2\otimes1}}\varphi
       =\int_{2\otimes1}\iota_{2\otimes1}(\varphi)\\
&=\int_{C_{1\otimes2}}\varphi=\int_{1\otimes2}\iota_{1\otimes2}(\varphi).
\end{align*}
Hence we have
\begin{align}\label{GHPR}
        \int_{2\otimes1}\iota_{2\otimes1}(\varphi)=\int_{1\otimes2}\iota_{1\otimes2}(\varphi).
\end{align}
\end{prop}
We call \eqref{GHPR} a generalized harmonic product relation.
This provides a relation of hyperlogarithms of the type $\cM_{0.5}$  which is defined as follows:
For $\alpha_i \in \{1,z_2\}\quad (i=1,\ldots,r)$, we set
\begin{align}
     L({}^{k_1}\alpha_1\cdots{}^{k_r}\alpha_r;z_1):&
            =\int_0^{z_1}\xi_1^{k_1-1}\omega_1\cdots\xi_1^{k_r-1}\omega_r \nonumber \\
            &= \sum_{n_1>\cdots>n_r>0}\!\!\!\!\!\!
            \frac{\alpha_1^{n_1-n_2}\alpha_2^{n_2-n_3}\cdots \alpha_r^{n_r}}{n_1^{k_1} 
            \cdots n_r^{k_r}} z_1^{n_1},
\end{align}
where $\ds \omega_i=\frac{\alpha_idz_1}{1-\alpha_iz_1}$. 
This is called hyperlogarithms of the type $\cM_{0.5}$ of the main variable $z_1$. 
If $\alpha_1=\cdots=\alpha_r=1$, it reduces to a MPL of one variable;
\begin{align}
            \Li_{k_1,\dots,k_r}(z_1)=L({}^{k_1}1\cdots{}^{k_r}1 \, ;z_1).
\end{align}
If $\alpha_1=\cdots=\alpha_i=1, \ \alpha_{i+1}=\cdots=\alpha_{i+j}=z_2 \ (r=i+j)$, 
it is a MPL of two variables;
\begin{align}\label{eq:MPL2}
     \Li_{k_1,\ldots,k_r}(i,r-i;z_1,z_2)
     =\sum_{n_1>n_2\cdots>n_r>0}\frac{z_1^{n_1}z_2^{n_{i+1}}}{n_1^{k_1}\cdots n_r^{k_r}}.
\end{align}

The generalized harmonic product relations are the main subject in \cite{OU}. 
Here we only give a simple example.

Let $\varphi=\xi_{11}\xi_{12}+\xi_{22}\xi_{11}-\xi_{22}\xi_{12}-\xi_2\xi_{12} \in \cB^0$. Then 
\begin{align*}
\iota_{1\otimes2}(\varphi)&=\xi_{11}\widetilde{\xi}_{12}^{(1)}\otimes \bnull,\\
\iota_{2\otimes1}(\varphi)&=\xi_{22}\otimes\xi_{11}-\xi_{22}\widetilde{\xi}^{(2)}_{12}\otimes\bnull-\xi_2\widetilde{\xi}^{(2)}_{12}\otimes\bnull,
\end{align*}
\begin{align*}
\int_{1\otimes2}\iota_{1\otimes2}(\varphi)&=\int_0^{z_1}\xi_{11}\widetilde{\xi}_{12}^{(1)}=\Li_{1,1}(1,1;z_1,z_2),\\
\int_{2\otimes1}\iota_{2\otimes1}(\varphi)&=\int_0^{z_2}\!\!\xi_{22}\int_0^{z_1}\!\!\xi_{11}-\int_0^{z_2}\!\!\xi_{22}\widetilde{\xi}^{(2)}_{12}-\int_0^{z_2}\!\!\xi_2\widetilde{\xi}^{(2)}_{12}\\
&=\Li_1(z_2)\Li_1(z_1)-\Li_{1,1}(1,1;z_2,z_1)-\Li_{2}(0,1:z_2,z_1).
\end{align*}
Thus \eqref{GHPR} for $\varphi$ reads 
\begin{align}\label{eq:hpr_11}
         \Li_1(z_1)\Li_1(z_2)=\Li_{1,1}(1,1;z_1,z_2)+\Li_{2}(0,1;z_2,z_1)+\Li_{1,1}(1,1;z_2,z_1).
\end{align}

From Proposition \ref{prop:GHPR}, we have 
\begin{align*}
\hcL_s(z_1,z_2)&=\int_{C_{1 \otimes 2}} 
              \left(\ad(\varOmega_0)+\mu(\varOmega')\right)^s(\bnull \otimes \bunit)\\
           &=\int_{C_{1 \otimes 2}} (\iota_{1 \otimes 2}\otimes \id_{\cU})
              \left(\left(\ad(\varOmega_0)+\mu(\varOmega')\right)^s(\bnull \otimes \bunit)\right),
\end{align*}
where $\hcL_s(z_1,z_2)$ is the homogeneous degree s part of the fundamental solution 
$\cL(z_1,z_2)$ of \ref{KZE2}. Hence we have the following proposition.

\begin{prop}[\cite{OU}, Corollary 21]\label{prop:itLs}
\begin{enumerate}
\item We have
\begin{align}\label{eq:itLs1}
     \hcL_s(z_1,z_2) = \sum_{s'+s''=s} \ \sum_{W',W''}
           \int_0^{z_1} \theta^{(1)}_{1 \otimes 2}(W')
              \int_0^{z_2} \ \theta^{(2)}_{1 \otimes 2}(W'') \ \alpha(W')\alpha(W'')(\bunit)
\end{align}
where $W'$ runs over the set $\cW_{s'}^0(X_1,X_{11},X_{12})$, and 
$W''$ runs over the set $\cW_{s''}^0(X_2,X_{22})$. \ 
($\cW^0_s(\fA)=\cW^0(\fA) \cap \cU_s(\fX)$, and $\cW^0(\fA)$
stands for the set of words of the letters $\fA$ ending with other than $X_1,X_2$.) 
$\alpha: \cU(\fX) \to \End(\cU(\fX))$ is an algebra homomorphism
\begin{align*}
    \alpha: (X_1,X_{11},X_2,X_{22},X_{12}) \mapsto 
                   (\ad(X_1),\mu(X_{11}),\ad(X_2),\mu(X_{22}),\mu(X_{12})),
\end{align*}
and 
\begin{align*}
   &\theta^{(1)}_{1 \otimes 2}:\cU(\C\{X_1,X_{11},X_{12}\}) \to S(\xi_1,\xi_{11},\widetilde{\xi}_{12}^{(1)}) \\
   &\theta^{(2)}_{1 \otimes 2}:\cU(\C\{X_2,X_{22}\}) \to S(\xi_2,\xi_{22})
\end{align*}
are linear maps defined by replacing 
\begin{align*}
     \theta^{(i)}_{1 \otimes 2}(X_i)=\xi_i, \ \theta^{(i)}_{1 \otimes 2}(X_{ii})=\xi_{ii} \ \ (i=1,2),
      \ \theta^{(1)}_{1 \otimes 2}(X_{12})=\widetilde{\xi}_{12}^{(1)}.
\end{align*}
\item Similarly we  have 
\begin{align}\label{eq:itLs2}
  \hcL_s(z_1,z_2) =\sum_{s'+s''=s} \ \sum_{W',W''}
           \int_0^{z_2} \theta^{(2)}_{2 \otimes 1}(W')
              \int_0^{z_1} \ \theta^{(1)}_{2 \otimes 1}(W'') \ \alpha(W')\alpha(W'')(\bunit)
\end{align}
where $W'$ runs over the set $\cW_{s'}^0(X_2,X_{22},X_{12})$, and $W''$ runs over the set 
$\cW_{s''}^0(X_1,X_{11})$, and 
\begin{align*}
  & \theta^{(2)}_{2 \otimes 1}:\cU(\C\{(X_2,X_{22},X_{12}\}) \to S(\xi_2,\xi_{22},\widetilde{\xi}_{12}^{(2)}) \\
  & \theta^{(1)}_{2 \otimes 1}:\cU(\C\{X_1,X_{11}\}) \to S(\xi_1,\xi_{11})
\end{align*}
are linear maps defined by replacing 
\begin{align*}
     \theta^{(i)}_{2 \otimes 1}(X_i)=\xi_i, \ \theta^{(i)}_{2 \otimes 1}(X_{ii})=\xi_{ii} \ \ (i=1,2),
      \  \theta^{(2)}_{2 \otimes 1}(X_{12})=\widetilde{\xi}_{12}^{(2)}.
\end{align*}
\end{enumerate}
\end{prop}

In this proposition, we should observe that the iterated integral  
\begin{align*}
     L\big( \theta^{(1)}_{1 \otimes 2}(W') \,;\,  z_1 \big) :=
     \int_0^{z_1} \theta^{(1)}_{1 \otimes 2}(W') \quad (W' \in \cW_{s'}^0(X_1,X_{11},X_{12}))
\end{align*}
is a hyperlogarithm of the type $\cM_{0,5}$ of the main variable $z_1$, 
and the iterated integral 
\begin{align*}
    L\big( \theta^{(2)}_{1 \otimes 2}(W'') \,;\, z_2 \big) := 
     \int_0^{z_2} \ \theta^{(2)}_{1 \otimes 2}(W'') \quad ( W''\in \cW_{s''}^0(X_2,X_{22}))
\end{align*}
is a multiple polylogarithm of the variable $z_2$.

\subsection{Five term relation of dilogarithms}

Let $\tau=(23)(45)$. Define an automorphism $\tau: \overline{\cM}_{0,5} \to \overline{\cM}_{0,5}$ as follows: 
For a point $P=[a_1,a_2,a_3,a_4,a_5]$, set 
\begin{align}
     \tau(P)=[a_{\tau^{-1}(1)},a_{\tau^{-1}(2)},a_{\tau^{-1}(3)},a_{\tau^{-1}(4)},a_{\tau^{-1}(5)}]. 
\end{align}
By this automorphism, the pentagon 1
is transformed to the pentagon 4 like as in the following figure. 

\vspace{20mm}
\hspace{20mm}
\begin{picture}(0,5)(0,-2)
\footnotesize
\put(0,0){\includegraphics{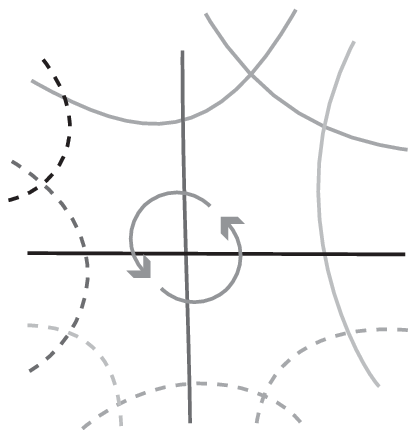}}
\put(4.2,1.7){$D_{23}$}
\put(1.7,-0.3){$D_{45}$}
\put(4.3,2.7){$D_{35}$}
\put(3.5,4){$D_{14}$}
\put(2.7,4.4){$D_{12}$}
\put(0,4){$D_{34}$}
\put(-0.5,2.7){$D_{15}$}
\put(3.1,-0.1){$D_{13}$}
\put(-0.4,1.1){$D_{24}$}
\put(4.3,0.9){$D_{25}$}
\put(2.5,2,7){$1$}
\put(1.1,1.1){$4$}
\put(1.2,2.4){$\tau$}
\put(2.4,1.1){$\tau$}
\end{picture}

\vspace{-10mm}

Let $\tau^*(x_i)=x_{\tau(i)} \quad (i=1,2,3,4,5)$, which induces an automorphism of the rational 
function field of $\overline{\cM}_{0,5}$, in particular, the cubic coordinates are transformed
as follows:
\begin{align}
    \tau^*(z_1,z_2) = \left(\frac{-z_1(1-z_2)}{1-z_1}, \, \frac{-z_2(1-z_1)}{1-z_2} \right).
\end{align}
Define the pull-back of $\omega_{ij}$ by $\tau^*(\omega_{ij})=d\log(\tau^*(x_i)-\tau^*(x_j))$,
and let $\tau_*: \fX \!\to\! \fX$ be an automorphism defined by 
\begin{align*}
     \sum_{1\le i<j\le 5}\omega_{ij}\tau_*(\varOmega_{ij})=\sum_{1\le i<j\le 5}\tau^*(\omega_{ij})\varOmega_{ij},
\end{align*}
namely, $\tau_*(\varOmega_{ij})=\varOmega_{\tau^{-1}(i),\tau^{-1}(j)}$.
Then we have
\begin{align*}
          \varOmega=(\tau^*\otimes \tau_*^{-1})(\varOmega)=\sum_{1\le i<j\le 5}
                                          \tau^*(\omega_{ij})\tau_*^{-1}(\varOmega_{ij}).
\end{align*}

Explicit representations of the pull-back $\tau^*$ and the induced automorphism $\tau_*$ are 
given as follows:
\begin{align}
\begin{cases}
     \tau^*(\xi_1)=\xi_1+\xi_{11}-\xi_{22}, \quad \tau^*(\xi_1)=-\xi_1+\xi_{12}, \\
     \tau^*(\xi_2)=-\xi_{11}+\xi_2+\xi_{22}, \quad \tau^*(\xi_{22})=-\xi_{22}+\xi_{12}, \quad \tau^*(\xi_{12})=\xi_{12},
\end{cases}
\end{align}
and 
\begin{align}
\begin{cases}
     \tau_*(X_1) = X_1, \quad  \tau_*(X_{11})=X_1-X_{11}-X_2, \quad \tau_*(X_2)=X_2 \\
     \tau_*(X_{22})=-X_1+X_2-X_{22}, \quad  \tau_*(X_{12})=X_{11}+X_{22}+X_{12}.
\end{cases} 
\end{align}

Since $\ds \varOmega=(\tau^*\otimes \tau_*^{-1})(\varOmega)$,
\begin{align*}
  \tcL(z_1,z_2)&=(\tau^*\otimes \tau_*^{-1})(\cL(z_1,z_2)) \\
               &=\cL(\tau^*(z_1),\tau^*(z_2))\Big|_{X \to \tau_*^{-1}X,\;\; X=X_1,X_{11},X_2,X_{22},X_{12}}
\end{align*}
is also a fundamental solution of \eqref{KZE2} which has the asymptotic condition
\begin{align}\label{eq:asymptotic}
\tcL(z_1,z_2) \sim \bunit \left(\frac{-z_1(1-z_2)}{1-z_1}\right)^{X_1} 
         \left(\frac{-z_2(1-z_1)}{1-z_2}\right)^{X_2} \quad  (z_1,z_2) \to (0,0).
\end{align}
Therefore the connection relation of $\cL(z_1,z_2)$ and $\tcL(z_1,z_2)$ is 
\begin{align}\label{connection_dilog1}
      \tcL(z_1,z_2) &= \cL(z_1,z_2) \exp(-\sgn(\Image z_1) \, \pi i X_1) \, 
      \exp(-\sgn(\Image z_2)\, \pi i X_2).
\end{align}
It is convenient to rewrite this formula as follows:
\begin{align}\label{eq:connection_dilog}
       (\tau^*\otimes \id)(\cL(z_1,z_2)) &= (\id\otimes \tau_*^{-1})(\cL(z_1,z_2)) \notag \\
       &\hspace{5mm} \times \exp(-\sgn(\Image z_1)\, \pi i X_1)\exp(-\sgn(\Image z_2)\, \pi i X_2).
\end{align}
Using \eqref{eq:itLs1} and \eqref{eq:itLs2} in Proposition \ref{prop:itLs}, we have the following proposition.

\begin{prop}
\begin{enumerate}
 \item From the coefficients of $[X_1,X_{11}]$ of the both sides of \eqref{eq:connection_dilog}, we have
\begin{align}\tag{L1}\label{eq:landen1}
    \Li_2\left(\frac{-z_1(1-z_2)}{1-z_1}\right)&= 
      \Li_{1,1}(1,1;z_1,z_2)-\Li_2(z_1)-\Li_{1,1}(z_1)  \\
     & \hspace{20mm}  +\Li_2(0,1;z_1,z_2). \nonumber 
\end{align}
\item From the coefficients of $[X_2,X_{22}]$ of both sides of \eqref{eq:connection_dilog}, we have
\begin{align}\tag{L2}\label{eq:landen2}
        \Li_2\left(\frac{-z_2(1-z_1)}{1-z_2}\right)
        &= -\Li_{1,1}(1,1;z_1,z_2)-\Li_2(z_2)-\Li_{1,1}(z_2) \\
        & \hspace{20mm} +\Li_1(z_2)\Li_1(z_1). \nonumber 
\end{align}
\end{enumerate}
\end{prop}
The formula \eqref{eq:landen1} is a two dimensional analogue of
Landen's formula of dilogarithms \cite{Le}:
\begin{align*}
         \Li_2\left(\frac{-z}{1-z}\right)=-\Li_2(z)-\frac{1}{2}\log^2(1-z).
\end{align*}

Since $\Li_2(0,1;z_1,z_2)=\Li_2(z_1z_2)$ and $\Li_{1,1}(z_1)=\frac{1}{2}\log^2(1-z_1)$, 
$\eqref{eq:landen1} + \eqref{eq:landen2}$ implies the following:

\begin{thm}
We have the five term relation for dilogarithms \cite{Le}:
\begin{align}\tag{5TR}\label{eq:5termsdilog}
\Li_2(z_1z_2) & =\Li_2\left(\frac{-z_1(1-z_2)}{1-z_1}\right) + \Li_2\left(\frac{-z_2(1-z_1)}{1-z_2}\right)
            +\Li_2(z_1)+\Li_2(z_2) \notag \\
          & \hspace{20mm} +\frac{1}{2}\log^2\left(\frac{1-z_1}{1-z_2}\right). \notag
\end{align}
\end{thm}

\end{document}